\documentclass[12pt]{amsart}
\usepackage{amssymb,enumerate}
\usepackage{graphicx}
\usepackage{subfig}
\newcommand{\D}{{\mathbb{D}}}

\newcommand{\X}{{\mathbb{R}^d}}

\newcommand{\Y}{{\mathbb{R}^d}}

\newcommand{\Leb}{{\mathcal{L}}}

\theoremstyle{plain}

\newtheorem{thm}{Theorem}[section]
\newtheorem{claim}[thm]{Claim}
\newtheorem{cor}[thm]{Corollary}
\newtheorem{defn}[thm]{Definition}
\newtheorem{ex}[thm]{Example}
\newtheorem{lem}[thm]{Lemma}
\newtheorem{propn}[thm]{Proposition}
\newtheorem{rmk}[thm]{Remark}

\begin{document}

\title {Optimal Transportation with Capacity Constraints}
\date{\today}

\author{Jonathan Korman}
\email{jkorman@math.toronto.edu}
%\address{Department of Mathematics, University of Toronto, Toronto Ontario M5S 2E4 Canada}

\author{Robert J. McCann}
\thanks{The second author is pleased to acknowledge the support of Natural Sciences and Engineering Research Council of Canada Grants 217006-08. \copyright 2011 by the authors. 
}

\email{mccann@math.toronto.edu}
\address{Department of Mathematics, University of Toronto, Toronto Ontario M5S 2E4 Canada}

%\subjclass{22E35, }
%\keywords{$p$-adic group, harmonic analysis, local character expansion}

%\numberwithin{equation}{section}

\begin{abstract}
The classical problem of optimal transportation can be formulated as a linear optimization problem on a convex domain: among all joint measures with fixed marginals find the optimal one, where optimality is measured against a cost function.
Here we consider a natural but largely unexplored
variant of this problem by imposing a pointwise
constraint on the joint (absolutely continuous) measures:
among all joint densities with fixed marginals and which are dominated by a given
density, find the optimal one.  For this variant, we show
local non-degeneracy of the cost function implies every minimizer is extremal
in the convex set of competitors, hence unique.
An appendix develops rudiments of a duality
theory for this problem, which allows us to compute several suggestive examples.
\end{abstract}

\maketitle

%\setcounter{section}{-1}
%\addtocounter{secnumdepth}{2}
\setcounter{secnumdepth}{2}

\section{Introduction}

The optimal transportation problem of Monge \cite{Monge81} and
Kantorovich \cite{Kantorovich42} has attracted much attention
in recent years; see the surveys
\cite{AmbrosioGigli11p} \cite{MG} \cite{Villani1} \cite{Villani2}.
However,  there is a variant of the problem which is almost as natural but
remains unexplored outside the discrete setting.  This variant, tackled below,
involves imposing capacity constraints which limit the amount transported
between any given source and corresponding sink.

Let $L^1_{c} (\mathbb{R}^n)$ denote the space of $L^1(\mathbb{R}^n)$-functions with compact
support, where $L^1$ is with respect to Lebesgue measure.
In this paper functions typically represent mass densities.
Given densities $0 \leq f, g \in L^1_{c} (\mathbb{R}^d)$ with same total mass
$\int f = \int g$, let  $\Gamma(f, g)$ denote the set of {\it joint densities}
$0 \leq h \in L^1_{c}( \X \times \Y)$ which have $f$ and $g$ as their marginals:
$f(x) = \int_{\Y} h(x,y)dy$ and $g(y) = \int_{\X} h(x,y)dx$. The set $\Gamma(f, g)$
is a convex set.\\

A {\it cost function} $c(x,y)$ represents the cost per unit mass for transporting material
from $x\in\X$ to $y\in\Y$. Given densities $0 \leq f, g\in L^1_{c}(\X)$ with same total mass,
and a cost $c(x,y)$, the problem of {\it optimal transportation} is to minimize the {\it transportation cost}
\begin{eqnarray}\label{transportation_cost}
I_c(h):=\int_{\X \times \Y} c(x,y)h(x,y)dx dy
\end{eqnarray}
among joint densities $h$ in $\Gamma(f, g)$, to obtain the optimal cost
\begin{eqnarray}\label{unconstrained: primal}
\underset{h\in\Gamma(f, g)}{\mbox{inf}} I_c(h).
\end{eqnarray}

In the context of transportation, a joint density $h\in\Gamma(f, g)$ can be thought of as representing a transportation plan.

In this paper we will sometimes refer to the traditional optimal
transportation problem as the {\it unconstrained} optimal transportation problem.\\

Given $0 \leq \overline{h}\in L^\infty(\X \times \Y)$ of compact support, we let $\Gamma(f, g)^{\overline{h}}$ denote the set of all $h\in\Gamma(f, g)$ {\it dominated} by $\overline{h}$, that is $h \leq \overline{h}$ almost everywhere. The set  $\Gamma(f, g)^{\overline{h}}$ is a convex set.\\

The optimization problem we will be concerned with in this paper---the {\it optimal transportation with capacity constraints}---is to minimize the transportation cost (\ref{transportation_cost})
among joint densities $h$ in $\Gamma(f, g)^{\overline{h}}$, to obtain the optimal cost under the {\it capacity constraint} $\overline{h}$
\begin{eqnarray}\label{eqn: constraind_optimal_cost}
\underset{h\in\Gamma(f, g)^{\overline{h}}}{\mbox{inf}} I_c(h).
\end{eqnarray}

\noindent{\bf Interpretation.} As an example of an optimal transportation problem in the
discrete case~\cite[Chapter 3]{Villani2}, consider a large number of bakeries producing
loaves of bread that should be transported (by donkeys) to caf\'es. The problem is to
find where each unit of bread should go so as to minimize the transportation cost.
The unconstrained optimal transportation problem assumes ideal donkeys that can
transport any amount of bread. The constrained version discussed here takes into
account the capacity limitations of the donkeys --- assuming of course
that each (cafe, bakery) pair has a donkey at its disposal,  and that no donkey services
more than one cafe and one bakery.

\begin{ex}\label{example} (Constrained optimal solution concentrates on `diagonal tiles in a $2 \times 2$ checker board' in response to an integer constraint.)

\begin{figure}[h!]
\centering
%%----start of first subfigure----
\subfloat[$h$]{
\label{fig:subfig:a}
%% label for first subfigure
\includegraphics[width=1in]{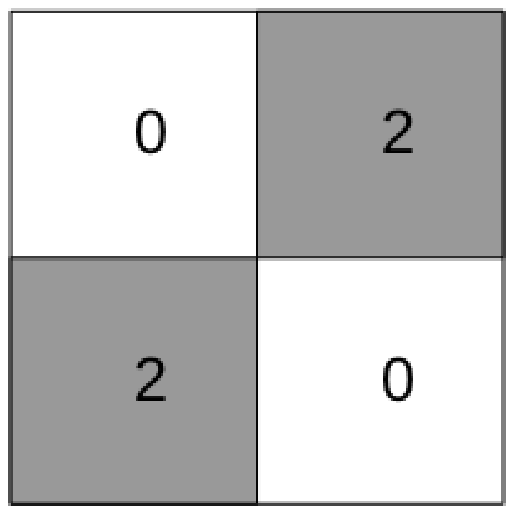}}
\hspace{1in}
%%----start of second subfigure----
\subfloat[$\overline{h}$]{
\label{fig:subfig:b}
%% label for second subfigure
\includegraphics[width=1in]{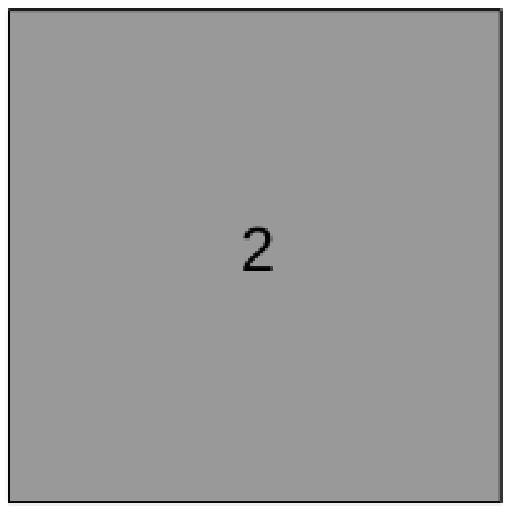}}
\caption{}
\label{fig:subfig}
%% label for entire figure
\end{figure}

Let $I$ be the closed interval $[-\frac{1}{2},\frac{1}{2}] \subset \mathbb{R}^1$ and let $                                                                                                                                                                                                                                                                                                                                                                                                                                                                                                                                                                                                                                                                                                                                                                                                                                                                                                                                                                                                                                                                                                                                                                                                                                                                                                                                                                                                                                                                                                                                                                                                                                                                                                                                                                                                                                                                                                                                                                                                                                                                                                                                                                                                                                                                                                                                                                                                                                                                                                                                                                      f=g=1_I$ have constant density $1$ on $I$ (here $1_I$ is the characteristic function of the set $I$). Let $\overline{h}=2 \cdot 1_{I^2}$ have constant density $2$ on $I^2$ (figure 1B). Note that $0 \leq f, g\in L^1_{c}(I)$ have same total mass $1$, and that $\Gamma(f, g)^{\overline{h}}\neq \emptyset$ since it contains $1_{I^2}$. Let $c(x,y)=\frac{1}{2}|x-y|^2$. Then, as explained in the appendix, $I_c(\cdot)$ attains its minimal value on $\Gamma(f, g)^{\overline{h}}$ at (see figure 1A)
\begin{eqnarray}\label{eqn:example}
h(x,y):= \left\{
\begin{array}{cc}
2 & \mbox{ on } [-\frac{1}{2},0] \times [-\frac{1}{2},0] \cup [ 0,\frac{1}{2} ] \times [0, \frac{1}{2} ] \\
0 & \mbox{ otherwise. }
\end{array}\right.
\end{eqnarray}
\end{ex}

Other examples can be derived from this one (see Remark~\ref{more_examples}).
Lest such examples seem obvious, we also pose the following open problem:\\

\begin{ex}\label{example2} (Open problem.)

\begin{figure}[h!]
\centering
%%----start of first subfigure----
\subfloat[$h$]{
\label{fig:subfig:a}
%% label for first subfigure
\includegraphics[width=1in]{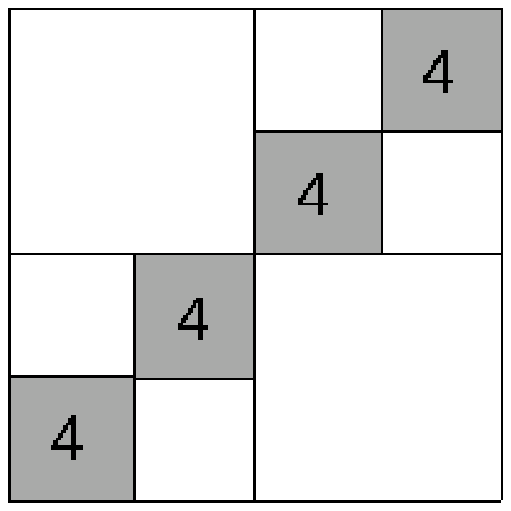}}
\hspace{1in}
%%----start of second subfigure----
\subfloat[$\overline{h}$]{
\label{fig:subfig:b}
%% label for second subfigure
\includegraphics[width=1in]{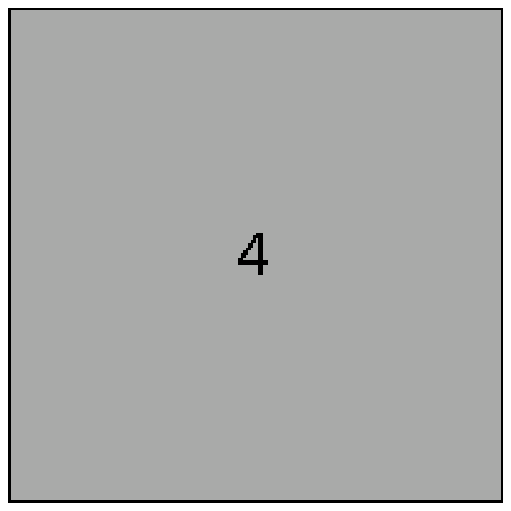}}
\hspace{1in}
%%----start of third subfigure----
\subfloat[$\Delta h$]{
\label{fig:subfig:c}
%% label for second subfigure
\includegraphics[width=1in]{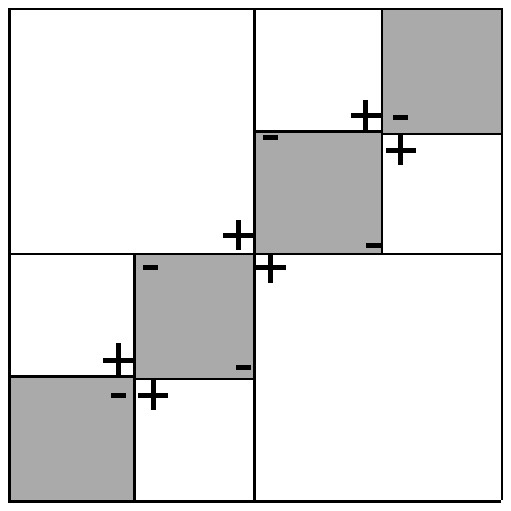}}
\caption{}
\label{fig:subfig}
%% label for entire figure
\end{figure}

Let $I, f, g$ and $c$ be as in example~\ref{example}. Let $\overline{h}=4 \cdot 1_{I^2}$ have constant density $4$ on $I^2$ (figure 2B). After considering example~\ref{example} it is natural to guess that $I_c(\cdot)$ attains its minimal value on $\Gamma(f, g)^{\overline{h}}$ at (see figure 2A)
\begin{eqnarray}\label{eqn:example}
h(x,y):= \left\{
\begin{array}{cc}
4 & \mbox{ on } S\\
0 & \mbox{ otherwise, }
\end{array}\right.
\end{eqnarray}

where $S:=[-\frac{1}{2},-\frac{1}{4}] \times [-\frac{1}{2},-\frac{1}{4}] \cup [-\frac{1}{4},0] \times [-\frac{1}{4}, 0] \cup [ 0,\frac{1}{4} ] \times [0, \frac{1}{4} ] \cup  [\frac{1}{4},\frac{1}{2}] \times [\frac{1}{4},\frac{1}{2}]$.
Surprisingly, this is not the case. The perturbation $\Delta h$ in figure 2C reduces the total cost of $h$. Here `$+$' represents adding $\delta$ mass and `$-$' subtracting $\delta$ mass. Since adding/subtracting mass near the diagonal has negligible cost the net contribution of $\Delta h$ is dominated by the four minuses near the four points $(-\frac{1}{4}, 0), (0, -\frac{1}{4}), (0, \frac{1}{4})$ and $(\frac{1}{4}, 0)$.
So $\Delta h$ strictly reduces the total cost of $h$. We don't know the true optimizer
for this example. \\
\end{ex}

\begin{ex}\label{example3} (Constrained optimal solution with respect to periodic cost concentrates on `diagonal strip'.)

\begin{figure}[h!]
\centering
%%----start of first subfigure----
\subfloat[$R'$]{
\label{fig:subfig:a}
%% label for first subfigure
\includegraphics[width=1.8in]{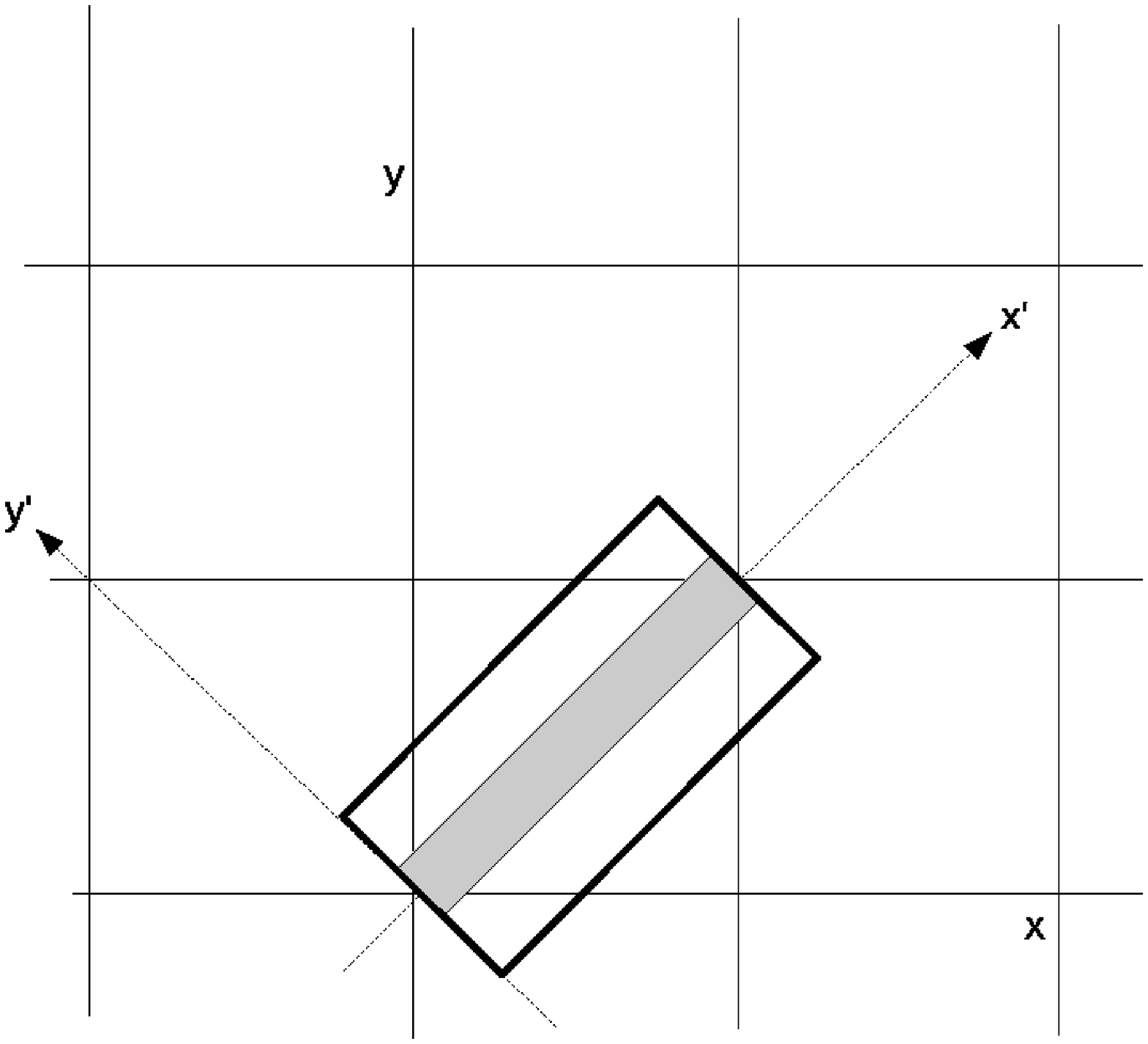}}
\hspace{1in}
%%----start of second subfigure----
\subfloat[$R$]{
\label{fig:subfig:b}
%% label for second subfigure
\includegraphics[width=1.8in]{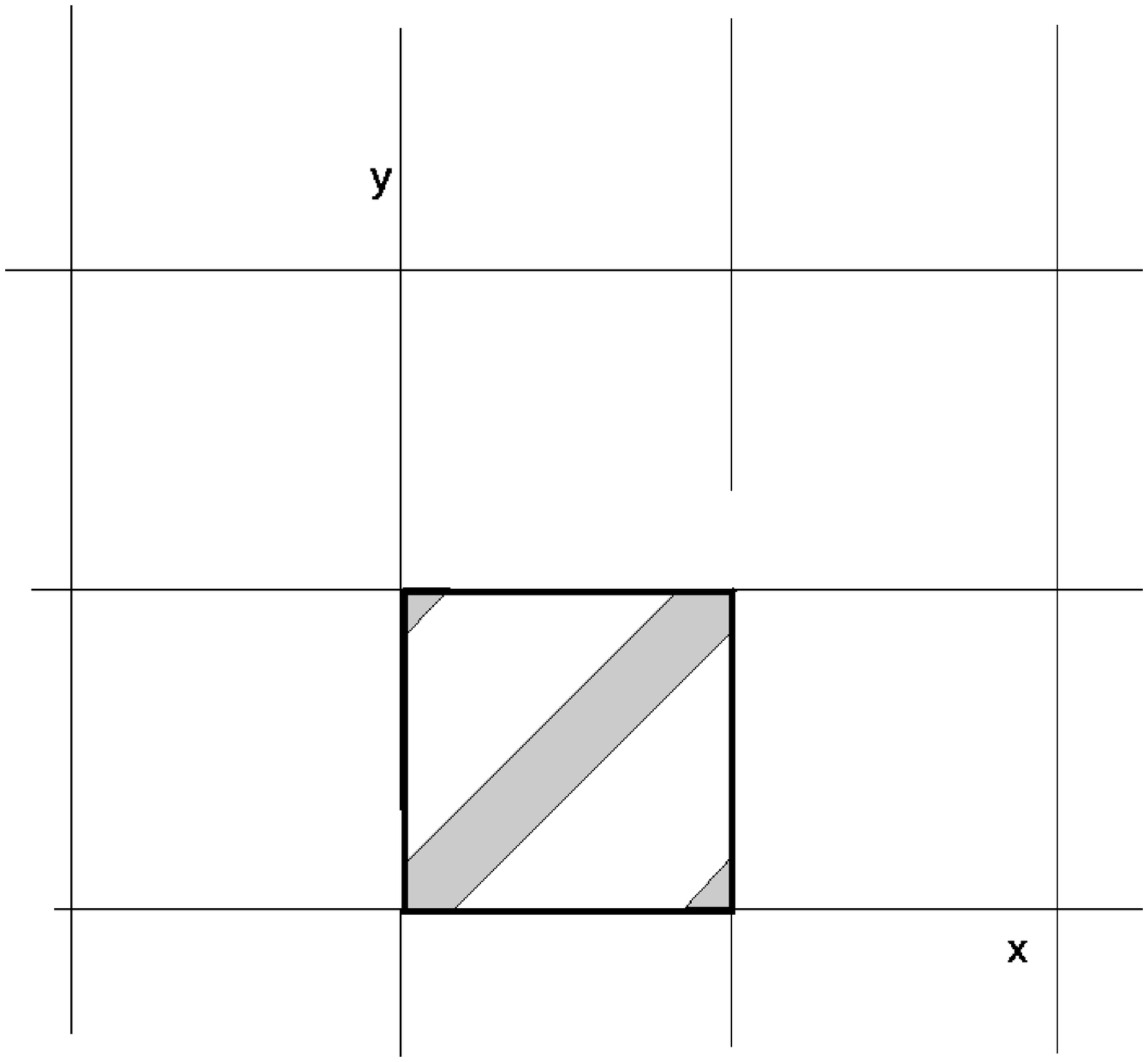}}
\caption{}
\label{fig:subfig}
%% label for entire figure
\end{figure}

Let $\mathbb{R}^2 / \mathbb{Z}^2$ be the periodic unit square, that is $\mathbb{R}^2$ where $(x,y)$ is identified with $(x',y')$ whenever $x-x', y-y' \in \mathbb{Z}$, and put the {\it periodic } cost function $c(x,y)=\underset{n\in \mathbb{Z}}{\mbox{inf}} |x-y-n|^2$ on it. Two fundamental domains are $R$ (see figure 3B), and $R'$ (see figure 3A).

The coordinate change $x':=y+x$, $y':=y-x$ maps $R$ bijectively onto a square of side-length $\sqrt{2}$, which can be identified with $R'$.
The cost becomes $c(x',y')=\underset{n\in \mathbb{Z}}{\mbox{inf}} |y'-n|^2$, which on $R'$ is just
$c(x',y')=y'^2$. Note that in the $x',y'$ coordinates, the cost is constant along lines parallel to $x'$. Given total mass $1$ and constant capacity bound $\overline{h} \geq 1$ on the periodic square, let
\begin{eqnarray}\label{eqn:example3}
h_0(x,y):= \left\{
\begin{array}{cc}
\overline{h} & \mbox{ on } S\\
0 & \mbox{ otherwise, }
\end{array}\right.
\end{eqnarray}

where $S$ is a diagonal strip in $R'$ of width $w=\frac{1}{\overline{h}\sqrt{2}}$ and length $\sqrt{2}$ centered about the diagonal $x'$ (see shaded strip in figure 3A). From the simple form of the cost in the $x',y'$ coordinates it can be easily seen that $h_0$ is the optimal way to fit mass $1$ into $R'$ while respecting the bound $\overline{h}$: $h_0= \underset{\underset{ mass(h)=1}{h \leq \overline{h}}}{\mbox{argmin}} \int_{R'} c h $. In particular $h_0= \underset{h \in \Gamma(1_I, 1_I)^{\overline{h}}}{\mbox{argmin}} \int_{R'} c h $, where $1_I$ is equal to the marginals of $h_0$. As a function on $R$, $h_0$ is supported on the shaded region in figure 3B.

Note that the uniqueness result, Theorem~\ref{thm:uniqueness}, still applies to this
cost, since it is $C^2$ and non-degenerate outside of two diagonal line segments on the periodic square.
\end{ex}

\noindent {\bf Motivation.} The thing to note from example~\ref{example}
is that at almost every point of the underlying space, the density $h$ of the optimal solution, is either equal to $0$ or to $\overline{h}$, the density of the capacity bound.
In the language developed below $h$ is geometrically extreme.

This example is special since the densities involved are both locally constant. It is easy to see that when $h$ and $\overline{h}$ are both constant in a neighbourhood of a point $(x_0,y_0)$, $h(x_0,y_0)$ must either equal $0$ or $\overline{h}(x_0,y_0)$: if  $0 < h(x_0,y_0) < \overline{h}(x_0,y_0)$ then a standard perturbation argument (see proof of Lemma~\ref{lem:necc_cond}) shows that $h$ cannot be optimal.

In general $\overline{h}$ and $h$ are not locally constant. But, one of the main insights
we exploit in this paper
is that, at an infinitesimal level they become constant: if we blow-up $\overline{h}$ and $h$ at a (Lebesgue) point, the blow-ups have constant densities (see (b) of Claim~\ref{claim:first}). In effect, blowing-up allows us to reduce the general case to the special case of locally constant densities, as is the case in example~\ref{example}.\\

\noindent {\bf Main result:  Existence and Uniqueness.} Proving solutions to the
capacity-constrained problem exist (Theorem~\ref{thm:existence}) requires very
minor modifications of the direct argument familiar from the unconstrained case.
The main result of this paper is therefore the uniqueness theorem
(Theorem~\ref{thm:uniqueness}). It says that under mild assumptions
on the cost and capacity bound, a solution to the capacity-constrained problem is unique.\\

\noindent {\bf Strategy.} Recall that a point of a convex set $\Gamma$ is called an {\it extreme point} if it is not an interior point of any line segment lying in $\Gamma$.
A density $h$ in $\Gamma(f, g)^{\overline{h}}$ will be called {\it geometrically extreme} (see Definition~\ref{defn:geometrically_extreme}) if there exists a (measurable) set $W\subset \mathbb{R}^{2d}$ such that $h(x,y)=\overline{h}(x,y)1_{W}(x,y)$ for almost every $(x,y)\in \mathbb{R}^{2d}$. (Such a density might be called `bang-bang' in the optimal control context). Observe that a density is an extreme point of $\Gamma(f, g)^{\overline{h}}$ if and only if it is geometrically extreme (with respect to $\overline{h}$).

It is well-known in the theory of linear programming that every continuous linear functional on a compact convex set attains its minimum at an extreme point. Our strategy for proving uniqueness in the problem at hand (Theorem~\ref{thm:uniqueness}) will be to show that {\em every} optimizer is geometrically extreme (Theorem~\ref{thm:optimality_ge}), hence is an extreme point of $\Gamma(f, g)^{\overline{h}}$. Since any convex combination of optimizers is again optimal (but fails to be geometrically extreme), it follows that no more than one optimizer exists.\\

\noindent {\bf Remark.} Once a solution is known to be geometrically extreme,
the entire problem is reduced to identifying the geometry of its support $W$.
Example \ref{example} shows the boundary of $W$ cannot generally be expected
to be smooth.  It is natural to wonder how to characterize $W$, and what kind
of geometric and analytic properties $\partial W$ will generally possess.\\

\noindent {\bf Main assumptions.} The main two assumptions for the uniqueness result
are that the capacity constraint $\overline{h}$ is uniformly bounded, and that the
cost $c(x,y)$ is non-degenerate (in the sense that $det D^2_{xy} c(x,y) \neq 0$ in
equation (\ref{Hessian})). Sufficiency of a local condition for uniqueness
is somewhat of a surprise; c.f. the cylindrical example of~\cite[p.10]{MPW}, which suggests
that --- except in one dimension --- no {\em local} hypothesis on the cost function
is sufficient to guarantee uniqueness of minimizer in the unconstrained case.\\

\noindent {\bf Remark.} Although capacity constraints are quite standard in the discrete
case, %(e.g.~\cite[not Hitchcock]{})
they do not seem to have been much considered in the continuum setting.
On the other hand, the work of Brenier~\cite{B87}\cite{B91} marks a turning point
in our understanding of unconstrained transportation in the continuum setting,
and we were surprised to discover that many of the insights gained in that context
do not seem to adapt easily to the capacity-constrained problem.\\

\noindent {\bf Acknowledgements}. The first author would like to thank Najma Ahmad for teaching him the basics of optimal transportation and introducing him to RJM. Example~\ref{example3} arose from a conversation with Yann Brenier.

\section{Notation, Conventions, and assumptions}\label{section:notation}

For a differentiable map $T:\mathbb{R}^n \rightarrow \mathbb{R}^m$ let $DT$ denote the derivative of $T$, that is the Jacobian matrix of all partial derivatives $\left(\frac{\partial T_i}{\partial x_j}\right)_{ 1 \leq i \leq m, 1 \leq j \leq n}$.\\

Let $D^2c(x,y)$ denote the Hessian of $c$ at the point $(x,y)$, that is the $2d \times 2d$ matrix of second order partial derivatives of the function $c$ at $(x,y)$. Let $D^2_{xy}c(x,y)$ denote the $d\times d$ matrix of mixed second order partial derivatives of $c$ at $(x,y)$, that is $\left(\frac{\partial^2 c}{\partial x_i \partial y_j}(x,y)\right)_{ 1 \leq i,j\leq d}$. Note that $D^2_{xy}c(x,y)$ is a sub matrix of the Hessian matrix:
\begin{eqnarray}\label{Hessian}
D^2c(x,y)=\left( \begin{array}{cc} D^2_{xx}c(x,y) & D^2_{xy}c(x,y) \\D^2_{yx}c(x,y) & D^2_{yy}c(x,y) \end{array} \right).
\end{eqnarray}
\\
The $n$-dimensional Lebesgue measure on $\mathbb{R}^n$ will be denoted by $\Leb^n$.\\

Let $\pi_X :\X\times\Y \longrightarrow \X : (x,y) \mapsto x$ and $\pi_Y :\X\times\Y \longrightarrow \X : (x,y) \mapsto y$ be the canonical projections.
For a density function $h \in L^1( \X \times \Y)$ denote its marginals by $h_X$ and $h_Y$: $h_X(x):= \int_{\Y} h(x,y)dy$ and $h_Y(y):= \int_{\X} h(x,y)dx$.

\subsection{Assumptions on the cost}\label{assumptions:cost} Consider the following assumptions on the cost:
\begin{itemize}
\item[(C1)] $c(x,y)$ is bounded,
\item[(C2)] there is a Lebesgue negligible closed set $Z \subset \X \times \Y$ such that $c(x,y)\in C^2(\X \times \Y \setminus Z)$ and,
\item[(C3)] $c(x,y)$ is {\it non-degenerate}: $det D^2_{xy}c(x,y)\neq 0$ for all $(x,y)\in \X \times \Y \setminus Z$.
\end{itemize}

\subsection{Assumptions on the capacity constraint}\label{assumptions:capacity} In section~\ref{section:existence} and from section~\ref{section:inherited} onwards, we will always assume that $\overline{h}$ is measurable and non-negative, has compact support, and is bounded on $\X \times \Y=\mathbb{R}^{2d}$. 
%We make no claims about other $\overline{h}$.

Given marginal densities $0\leq f, g\in L^1_{c}(\X)$ with same total mass, to avoid talking about the trivial case, we will always assume that a feasible solution exists: $\Gamma(f, g)^{\overline{h}}\neq \emptyset$.

\begin{rmk} To guarantee that the transportation cost $I_c(h)$ is finite we require $\overline{h}$ to have compact support: since the cost $c$ is always assumed continuous and bounded,  $\overline{h}$ having compact support makes sure that $I_c(h) \leq I_c(\overline{h}) <\infty$ for all $h \leq \overline{h}$. Note that when $\overline{h}$ has compact support, so will any density in $\Gamma(f, g)^{\overline{h}}$, as well as $f$ and $g$.
\end{rmk}

\section{Existence}\label{section:existence}

For simplicity we prove existence only in the case when $\overline{h}$ has compact support.

\begin{thm}\label{thm:existence} (Existence)
Assume that the cost $c$ is continuous and bounded. Take $0 \leq \overline{h}\in L^\infty(\X \times \Y)$ of compact support and let $0 \leq f, g \in L^1_{c}(\X)$ be marginal densities
for which $\Gamma(f, g)^{\overline{h}}\neq \emptyset$.
Then the corresponding problem of optimal transportation with capacity constraints (\ref{eqn: constraind_optimal_cost}) has a solution. That is, $I_c(\cdot)$ attains its minimum value on $\Gamma(f, g)^{\overline{h}}$.
\end{thm}

\begin{proof}
Let $X, Y \subset \X$ be compact subsets such that $spt(\overline{h})\subset X \times Y$.
Note that the support of any $h\in\Gamma(f, g)^{\overline{h}}$ is also contained in $X \times Y$, and that $spt(f) \subset X$, $spt(g) \subset Y$.

Since $\overline{h}$ is bounded and has compact support, $\overline{h}\in L^p(X \times Y)$ for all $1 \leq p \leq \infty$, in particular $\overline{h}\in L^2(X \times Y)$. Consequently $\Gamma(f, g)^{\overline{h}} \subset L^2(X \times Y)$.

We shall now specify a topology on $L^2 (X\times Y)$ for which $\Gamma(f, g)^{\overline{h}}$ is compact and $I_c(\cdot)$ continuous. Existence then follows from the general fact that a continuous function attains it minimum on a compact set. For $X$ and $Y$ compact, it is convenient to use the weak-$*$ topology, as in the unconstrained transportation problem (e.g.~\cite{Villani1}). Since $L^2$ is reflexive, the weak-$*$ topology is the same as the weak topology.
For the sake of completeness, we outline the direct argument despite its standard nature.

Give $L^2 (X\times Y)$ the weak topology. By the Banach-Alaoglu Theorem any closed ball $B_r(0)$ of radius $r < \infty$ in $L^2 (X\times Y)$ is weak-$*$, hence weak, compact.
Note that any $h$ with $0\leq h\leq \overline{h}$ satisfies $||h||_2 \leq ||\overline{h}||_2 =:R < \infty$.  Hence $\Gamma(g, h)^{\overline{h}}$ is contained in $B_R(0)$. So in order to show that $\Gamma(f, g)^{\overline{h}}$ is compact, it is enough to show that it is closed.

Let $h_n$ be a sequence in $\Gamma(f, g)^{\overline{h}}$ which converges weakly to $h_\infty \in B_R(0)$. We want to show  $h_\infty \in \Gamma(f, g)^{\overline{h}}$, that is that $h_\infty$ is dominated by $\overline{h}$ almost everywhere and has $f$ and $g$ as marginals.

Weak convergence means that for all $\psi\in L^2(X \times Y)$,
\begin{equation}\label{weak_convergence}
\lim_{n \to \infty} \int_{X \times Y} h_n(x,y)\psi(x,y) = \int_{X \times Y} h_\infty(x,y)\psi(x,y).
\end{equation}

Since $h_n \leq \overline{h}$, $\int h_n \psi \leq \int \overline{h} \psi$ for all non-negative $\psi\in  L^2$. Letting $n\rightarrow \infty$, $\int h_\infty \psi \leq \int \overline{h} \psi$ for all non-negative $\psi\in L^2$, hence $h_\infty \leq \overline{h}$ almost everywhere.

It is easy to see that $(h_\infty)_X=f$ by using the definition of weak convergence (\ref{weak_convergence}) with $\psi(x,y):=\psi(x)1_{Y}(y)$, where $\psi \in L^2(X)$. A similar calculation shows that $(h_\infty)_{Y}=g$. It follows that $\Gamma(f, g)^{\overline{h}}$ is weakly closed.

To see $I_c(\cdot): \Gamma(f, g)^{\overline{h}} \rightarrow \mathbb{R}$ is continuous with respect to the weak topology, use equation (\ref{weak_convergence}) with $\psi(x,y):=c(x,y)1_{X \times Y}(x,y)$, which is in $L^2(X \times Y)$ since $c$ is assumed bounded, to conclude that

\begin{equation*}
I_c (h_\infty )=\int h_\infty c = \lim_{n \to \infty} \int h_n c = \lim_{n \to \infty} I_c(h_n).
\end{equation*}

Existence in the constrained case follows.

\end{proof}

\section{Blowing up a density near a Lebesgue point}\label{section:blow-up}

When $0\le h$ is dominated by $\overline{h}\in L^\infty$ it is also bounded.
Even when $\overline{h}$ is continuous, $h \in L^1$ may not be continuous as we have seen
in example~\ref{example}; however it is necessarily measurable, belonging to $L^1$.
The notion of a Lebesgue point is a substitute for the notion of a point of continuity in the measure theoretic context. In this section we study the behaviour of $h$ near its Lebesgue points.\\

Given a Lebesgue point $(x_0,y_0)\in \X \times \Y$ of $0 \leq k\in L^1_{c}(\X\times\Y)$, consider the constant function $k_{\infty}(x,y):=k(x_0,y_0)$ defined on the unit cube $Q:=[-\frac{1}{2},\frac{1}{2}]^{d}\times [-\frac{1}{2},\frac{1}{2}]^{d}$. We call $k_\infty$ the blow-up of $k$ at $(x_0,y_0)$.\\

Let $Q_n=Q_n(x_0,y_0):=(x_0,y_0)+[-\frac{1}{2n},\frac{1}{2n}]^{d}\times [-\frac{1}{2n},\frac{1}{2n}]^{d}$ denote small cubical neighbourhoods of volume $(\frac{1}{n})^{2d}$ centered at $(x_0,y_0)\in \X\times\Y$.

Let $\varphi_n:Q\rightarrow Q_n\subset \X\times\Y$ be given by $\varphi_n(x,y)=(x_0,y_0)+\frac{1}{n}(x,y)$. Let $k_n:Q\rightarrow \mathbb{R}$ be defined by
\begin{equation}\label{blow_up_seq}
k_n:=k \circ \varphi_n.
\end{equation}

It will follow from claim \ref{claim:first} that $k_n$ converges to $k_\infty$ strongly in $L^1(Q)$.

\begin{defn} We call $k_n$ the {\it blow-up sequence} of $k$ at $(x_0,y_0)$. We call its limit $k_\infty$, the {\it blow-up} of $k$ at $(x_0,y_0)$.
\end{defn}

We recall some basic facts about Lebesgue points from~\cite{Rudin}.\\
Let $f\in L^1(\mathbb{R}^n)$. Any $x\in \mathbb{R}^n$ for which it is true that
$$\underset{r \rightarrow 0}{lim} \frac{1}{\Leb^n [B_r(x)]}\int_{B_r(x)}|f(y)-f(x)|dy=0$$
is called a {\it Lebesgue point} of $f$. Here $B_r(x)$ denotes the open ball with center $x$ and radius $r>0$. At a Lebesgue point $x$, an $L^1$-function $f$ has a well defined value:
$$f(x)=\underset{r \rightarrow 0}{lim} \frac{1}{\Leb^n [R_r(x)]}\int_{R_r(x)}f(y)dy.$$
Here $\{R_r(x)\}$ is any sequence of sets which `shrink nicely' to $x$ (e.g. cubes, spheres).

If $x$ is a point of continuity of $f$ then $x$ is a Lebesgue point of $f$. In particular, for a continuous function, every point is a Lebesgue point. Given $f\in L^1(\mathbb{R}^n)$, Lebesgue's Theorem says that almost every point in $\mathbb{R}^{n}$ is a Lebesgue point of $f$.

\begin{claim}\label{claim:first}
Let $(x_0,y_0)$ be a Lebesgue point of $0 \leq k\in L^1_{c}(\X \times \Y)$. Let $k_n$ denote the blow-up sequence of $k$ at $(x_0,y_0)$ and let $k_\infty$ denote the blow-up of $k$ at $(x_0,y_0)$. Then:
\begin{enumerate}
\item[(a)] {$k_n \rightarrow k_\infty \mbox{ strongly in } L^1(Q)$, i.e. $||k_n-k_\infty||_{L^1(Q)}\rightarrow 0$,}
\item[(b)] {$k_n(x,y)=k(\frac{1}{n} x+x_0, \frac{1}{n} y+y_0)$ on $Q$.}
\end{enumerate}
\end{claim}

\begin{proof}
(a) Letting $\varphi_n$ denote the dilation from \eqref{blow_up_seq} yields
\begin{eqnarray*}
\int_{Q} | k_n - k_\infty | dx dy &=& \int_{Q} | (k - k(x_0,y_0)) \circ \varphi_n | dx dy\\
&=& \frac{1}{\Leb^{2d}[Q_n]} \int_{Q} | (k - k(x_0,y_0)) \circ \varphi_n ||det D \varphi_n | dx dy\\
&=& \frac{1}{\Leb^{2d}[Q_n]} \int_{Q_n} | (k(x,y) - k(x_0,y_0))| dx dy \rightarrow 0,
\end{eqnarray*}
as $n \rightarrow \infty$.
The first equality is the definition of $k_n$, and the second equality uses $|det D\varphi_n(x,y)|=\left| \begin{array}{cc} \frac{1}{n^d} & 0 \\ 0  & \frac{1}{n^d} \end{array} \right|=\frac{1}{n^{2d}}=\Leb^{2d}[Q_n]$. The last equality follows from the change of variable formula and the limit at the end follows from $(x_0,y_0)$ being a Lebesgue point of $k$.

(b) follows immediately from the definition of $\varphi_n$.
\end{proof}

For later use we record the following immediate consequence of above claim.

\begin{rmk}\label{cor:blow-up}
Let $0 \leq \overline{h}\in L^\infty(\X \times \Y)$ have compact support. Suppose that $0 \leq h \leq \overline{h}$ and that $(x_0,y_0)$ is a common Lebesgue point of $h$ and $\overline{h}$.
Then, letting $h_n$ and $\bar h_n$ denote the blow-up sequences of $h$ and $\bar h$ at $(x_0,y_0)$,
\begin{enumerate}
\item[(a)] {$h_{\infty}(x,y)= h(x_0,y_0)$ on $ Q$ and $|| h_n - h_\infty||_{L^1(Q)}\rightarrow 0$, and}
\item[(b)] {$\overline{h}_{\infty}(x,y)= \overline{h}(x_0,y_0)$ on $Q$ and $|| \overline{h}_n - \overline{h}_\infty||_{L^1(Q)}\rightarrow 0$.}
\end{enumerate}
\end{rmk}

The following proposition clarifies
the nature of convergence of $h_n$ on $Q$.
It says that (for a subsequence $n(i)$) $Q$ can be partitioned into a `good' set,
$\widetilde{F}_n$, and a `bad' set, $\widetilde{E}_n$.
On the good sets $h_n$ converges `uniformly' to $h(x_0,y_0)$ while on
the bad sets it is uniformly bounded; the good sets are large and the bad are small.
Recall that the function $\overline{h}$ is assumed to be bounded and that $Q_1$ is compact.

\begin{propn}\label{propn:first}
Let $0 \leq \overline{h}\in L^1 \cap L^\infty(\X \times \Y)$.
Suppose that $0 \leq h \leq \overline{h}$ almost everywhere, and let $h_n$
denote the blow-up sequence of $h$ at a Lebesgue point $(x_0,y_0)$.
For some subsequence indexed by $n \in N_0 = \{n_1 < n_2< \cdots \}$
there exist non-negative real numbers $\alpha_n \rightarrow 0$,
and Borel subsets $\widetilde{E}_n$ and
$\widetilde{F}_n:=Q \setminus \widetilde{E}_n$ of $Q$,
such that
\begin{enumerate}
\item[(a)] {$0 \leq \Leb^{2d}[\widetilde{E}_n]\leq \alpha_n$,}
% and $1 - \alpha_n \leq \Leb^{2d}[\widetilde{F}_n] \leq 1$,
\item[(b)] {$||h_n(x,y)-h(x_0,y_0)||_{L^\infty(\widetilde{F}_n)}\leq \alpha_n$,}
\item[(c)] {$|h_n(x,y)|\leq \|\overline{h}\|_{L^\infty(Q_1)}$ for almost every $(x,y)\in Q$.}
\end{enumerate}
\end{propn}

\begin{proof}[Proof of (a)-(b).]

By Remark~\ref{cor:blow-up}, $h_k \rightarrow h_\infty=h(x_0,y_0)$ strongly in $L^1(Q)$,
i.e.\ $||h_k - h_\infty||_{L^1(Q)} \rightarrow 0$.
It follows that a subsequence $h_{k_i}$ converges pointwise to $h_\infty$ almost
everywhere on $Q$; for example, choosing
$||h_{k_i} - h_{k_{i+1}}||_{L^1(Q)} \leq \frac{1}{2^i}$
is known to assure this \cite[Theorem 2.7]{LL01}.
By Egoroff's Theorem, for any natural number $m$, there exists an open subset $\widetilde{E}_m \subset Q$ such that $0 \leq \Leb^{2d}[\widetilde{E}_m] \leq \frac{1}{m}$ and $||h_{k_i} - h_\infty||_{L^{\infty}(\widetilde{F}_m)} \rightarrow 0$ as $i \rightarrow \infty$, where $\widetilde{F}_m:=Q \setminus \widetilde{E}_m$. Hence, for $i=i_m$ large enough,
\begin{equation}\label{eqn:unif_cnv}
||h_{k_{i_m}} - h_\infty||_{L^{\infty}(\widetilde{F}_m)} < \frac{1}{m}.
\end{equation}

Note that without loss of generality we can assume that $k_{i_1} < k_{i_2} < k_{i_3} < \cdots $.
Let $N_0:=\{k_{i_m}\, | \, m\in \mathbb{N} \}$. Relabeling indices by $n\in N_0$: $n:=k_{i_m}$, $\widetilde{E}_n:=\widetilde{E}_m$, and letting $\alpha_n:=\frac{1}{m}$, the above equation becomes, for all $n\in N_0$,
\begin{equation}\label{eqn:unif_cnv2}
||h_n - h_\infty||_{L^{\infty}(\widetilde{F}_n)} < \alpha_n.
\end{equation}

{\em Proof of (c).} For almost every $(x,y) \in Q$ and all $n \in \mathbb{N}$ we have by (b) of Claim~\ref{claim:first}: $h_n(x,y)=h(\frac{1}{n}x+x_0, \frac{1}{n}y+y_0) \leq \overline{h}(\frac{1}{n}x+x_0, \frac{1}{n}y+y_0)\leq \|\overline{h}\|_{L^\infty(Q_1)}$.
\end{proof}

We also need a similar but more delicate result concerning convergence
of the marginals of $h_n$.
Recall that $Q=[-\frac{1}{2},\frac{1}{2}]^{d} \times [-\frac{1}{2},\frac{1}{2}]^{d}$.

\begin{propn}\label{propn:second}
Let $h_n$ be the blow-up sequence of $0 \leq h \in L^1\cap L^\infty(\X \times \Y)$ at a Lebesgue point $(x_0,y_0)$ and let $f_n:=(h_n)_X$ and $g_n:=(h_n)_Y$ be the corresponding marginals.
Taking $N_0$ (see Proposition~\ref{propn:first}) smaller if necessary yields a further subsequence indexed by $n \in N_0$,
with Borel subsets
$\widetilde{X}_n^{bad}, \widetilde{Y}_n^{bad} \subset
[-\frac{1}{2},\frac{1}{2}]^{d}$ and
$\widetilde{X}_n^{good}:= [-\frac{1}{2},\frac{1}{2}]^{d} \setminus \widetilde{X}_n^{bad}$
and
$\widetilde{Y}_n^{good}:= [-\frac{1}{2},\frac{1}{2}]^{d} \setminus \widetilde{Y}_n^{bad}$
such that,
\begin{enumerate}
\item[(a)] {$\lim\limits_{n \to \infty}{\Leb^{d}[\widetilde{X}_n^{bad}]} =0$ and  $\lim\limits_{n \to \infty}\Leb^{d}[\widetilde{Y}_n^{bad}]=0$;}
\item[(b)] {$\displaystyle \lim_{n \to \infty} \|f_n-h(x_0,y_0)\|_{L^\infty (\widetilde{X}_n^{good})} =0
=  \lim_{n \to \infty}
\| g_n- h(x_0,y_0) \|_{L^\infty (\widetilde{Y}_n^{good})}$;}
\item[(c)] {If $0 \leq h \leq \overline{h}$ as in Proposition~\ref{propn:first}, then $f_n\leq \|\overline{h}\|_{L^\infty(Q_1)}$ and $g_n\leq \|\overline{h}\|_{L^\infty(Q_1)}$
on $[-\frac{1}{2},\frac{1}{2}]^{d}$.}
\end{enumerate}
\end{propn}

\begin{proof}[Proof of (a)-(b)]
Let us start with the subsequence $h_n$ from Proposition~\ref{propn:first}.
Its marginals $f_n$ and $f_\infty$ are given by
$f_n(x):=\int_{[-\frac{1}{2},\frac{1}{2}]^{d}} h_n(x,y)dy$ and
$f_\infty(x):=\int_{[-\frac{1}{2},\frac{1}{2}]^{d}} h_\infty(x,y)dy=h(x_0,y_0)$.
The marginals $g_n$ and $g_\infty$ are defined similarly.

By (a) of Claim~\ref{claim:first}, $||h_k - h_\infty||_{L^1(Q)} \rightarrow 0$. It follows that $||f_k - f_\infty||_{L^1([-\frac{1}{2},\frac{1}{2}]^{d})} \rightarrow 0$ and $||g_k - g_\infty||_{L^1([-\frac{1}{2},\frac{1}{2}]^{d})} \rightarrow 0$. Let $f_{k_i}$ and $g_{k_i}$ be subsequences satisfying $||f_{k_i} - f_{k_{i+1}}||_{L^1([-\frac{1}{2},\frac{1}{2}]^{d})} \leq \frac{1}{2^i}$ and
$||g_{k_i} - g_{k_{i+1}}||_{L^1([-\frac{1}{2},\frac{1}{2}]^{d})} \leq \frac{1}{2^i}$.

As in the proof of Proposition~\ref{propn:first}, Theorem 2.7 of~\cite{LL01} and
Egoroff's Theorem imply existence of open subsets
$\widetilde{X}^{bad}_m, \widetilde{Y}^{bad}_m \subset [-\frac{1}{2},\frac{1}{2}]^{d}$
($m\in \mathbb{N}$) satisfying
$0 \leq \Leb^{2d}[\widetilde{X}^{bad}_m], \Leb^{2d}[\widetilde{Y}^{bad}_m]
\leq \frac{1}{m}$, such that for $i=i_m$ large enough,
$||f_{k_{i_m}} - f_\infty||_{L^{\infty}(\widetilde{X}^{good}_m)},
||g_{k_{i_m}} - g_\infty||_{L^{\infty}(\widetilde{Y}^{good}_m)}< \frac{1}{m}.$
By relabeling indices, as in the proof of Proposition~\ref{propn:first}, we get for all
$n$ in an index set $N_0$:
$0 \leq \Leb^{2d}[\widetilde{X}^{bad}_n],\Leb^{2d}[\widetilde{Y}^{bad}_n]  \leq \alpha_n$
and
$||f_n - f_\infty||_{L^{\infty}(\widetilde{X}^{good}_n)},
||g_n - g_\infty||_{L^{\infty}(\widetilde{Y}^{good}_n)}  < \alpha_n.$
(a) and (b) follow since $\alpha_n \rightarrow 0$ as $n \rightarrow \infty$.

{\em Proof of (c).} Follows immediately from (c) of
Proposition~\ref{propn:first} and the formula
$f_n(x)=\int_{[-\frac{1}{2},\frac{1}{2}]^{d}} h_n(x,y) dy$.
\end{proof}

\section{Optimality is inherited by blow-up sequence}\label{section:inherited}

When $h$ is optimal among densities which share its marginals and which are dominated by
$\overline{h}$, i.e. $h \in \underset{k\in\Gamma(h_X,h_Y)^{\overline{h}}}{\mbox{\textnormal{argmin}}} I_c(k)$, we show that $h_n$  is (almost) optimal among densities which share its marginals and which are dominated by $\overline{h}_n$, i.e.
$h_n \in \underset{k\in\Gamma((h_n)_X,(h_n)_Y)^{\overline{h}_n}}
{\mbox{\textnormal{argmin}}} I_{\widetilde{c}} (k)$.\\

We first record what conditions $(C2)-(C3)$ of subsection~\ref{assumptions:cost} on the cost imply about the Taylor expansion of $c$. Suppose the first and second derivatives of $c(x,y)$ exist at $(x_0, y_0)$ and
consider the $2^{nd}$-order Taylor expansion of $c$ near $(x_0,y_0)$:
\begin{eqnarray}\label{Taylor}
c(x_0+\frac{x}{n}, y_0+\frac{y}{n}) &=& c(x_0,y_0)+\sum^{d}_{i=1}\frac{\partial c}{\partial x_i}(x_0,y_0)\frac{x_i}{n} + \sum^{d}_{i=1}\frac{\partial c}{\partial y_i}(x_0,y_0)\frac{y_i}{n}\nonumber\\
&+& \frac{1}{n^2}\{\frac{1}{2}x^T D^2_{xx}c(x_0,y_0)x +\frac{1}{2}y^T D^2_{yy}c(x_0,y_0)y\\
&+& x^T D^2_{xy}c(x_0,y_0)y+ R_2(\frac{x}{n},\frac{y}{n})\}. \nonumber
\end{eqnarray}

Here $R_2$ is $n^2$ times the $2^{nd}$-order Lagrange remainder $R'_2(\frac{x}{n},\frac{y}{n})$ which satisfies $\|R'_2(\frac{x}{n},\frac{y}{n})\|_{L^\infty(Q)}=o(\frac{1}{n^2})$ (e.g. see~\cite[Theorem 19.1]{Spivak} for the $1$-dimensional case). Hence $\|R_2(\frac{x}{n},\frac{y}{n})\|_{L^{\infty}(Q)}=o(1)$ as $n\rightarrow \infty$.\\

When $D^2_{xy}c(x_0,y_0)$ is non-degenerate, changing the $y$ coordinates by $y_{new}=D^2_{xy}c(x_0,y_0) y_{old}$ gives, without loss of generality, that $D^2_{xy}c(x_0,y_0)= I$. Hence without loss of generality we can assume that $x^T D^2_{xy}c(x_0,y_0)y$, the mixed $2^{nd}$-order term in equation (\ref{Taylor}), is equal to $\widetilde{c}(x,y):= x \cdot y$. In other words, after an appropriate change of coordinates equation (\ref{Taylor}) assumes the form:

\begin{eqnarray}\label{Taylor_simplified}
& &c(x_0+\frac{x}{n}, y_0+\frac{y}{n}) =  \mbox{ constant term } + \mbox{ terms involving $x$ alone}\\
& & + \mbox{ terms involving $y$ alone} +  \frac{1}{n^2}\{ \widetilde{c}(x,y) + \widetilde{R}_2(\frac{x}{n},\frac{y}{n})\}. \nonumber
\end{eqnarray}

For $\widetilde{c}_n(x,y):=\widetilde{c}(x,y)+ \widetilde{R}_2(\frac{x}{n},\frac{y}{n})$ let $I_{\widetilde{c}_n}(k)$ denote
$\int_Q \widetilde{c}_n(x,y) k(x,y)$, and for $\widetilde{c}(x,y)= x \cdot y$ let $I_{\widetilde{c}}(k)$ denote $\int_Q \widetilde{c}(x,y) k(x,y)$.
Note that $I_{\widetilde{c}_n}(k)=\int_{Q}\widetilde{c}k + \int_{Q}\widetilde{R}_2 k$ and $|\int_{Q}\widetilde{R}_2 k| \leq \int_{Q}|\widetilde{R}_2| |k| \leq \|\widetilde{R}_2\|_{L^{\infty}(Q)} \|k\|_{L^1(Q)}$. Hence given a fixed constant $M > 0$, we have for all $k \in L^1(Q)$ whose total mass $\|k\|_{L^1(Q)}\leq M$,
\begin{equation}\label{eqn:o(1)}
I_{\widetilde{c}_n}(k)= I_{\widetilde{c}}(k) + o(1).
\end{equation}

\begin{rmk}
Note that when the cost satisfies $(C2)-(C3)$ of subsection~\ref{assumptions:cost}, for every $(x_0,y_0)\in \X \times \Y\setminus Z$ the first and second derivatives of $c(x,y)$ exist at $(x_0, y_0)$ and $D^2_{xy}c(x_0,y_0)$ is non-degenerate.
\end{rmk}

In~\cite[Theorem 4.6]{Villani2} it is shown that unconstrained optimality is inherited by restriction to (measurable) subsets: if the restricted plan is not optimal, then it can be improved, but any improvement in the restricted plan carries over to an improvement in the original optimal plan, which is not possible. In the constrained context, optimality is not necessarily inherited by an arbitrary restriction. To see this, recall example~\ref{example}, where the optimal constrained solution is given by $h$ in equation (\ref{eqn:example}). Note that the restriction of $h$ to $[0,\frac{1}{4}] \times [\frac{1}{4},\frac{1}{2}] \cup [\frac{1}{4},\frac{1}{2}] \times [0,\frac{1}{4}]$ is not optimal: restricting $h$ to $[0,\frac{1}{4}] \times [0,\frac{1}{4}] \cup [\frac{1}{4},\frac{1}{2}] \times [\frac{1}{4},\frac{1}{2}]$ has the same marginals but lower cost.\\

The following lemma says that in the constrained case, optimality is inherited when the restriction is to a {\it rectangular set}. This is used in the proof of Proposition~\ref{propn:third}.

\begin{lem}\label{lem:restriction} Let $0 \leq h\in L^1_{c} (\X \times \Y) $ be optimal among densities in $\Gamma (h_X, h_Y)^{\overline{h}}$ with respect to a cost function $c$. Consider a rectangular neighbourhood $A \times B \subset \X \times \Y$ where $A$ and $B$ are Borel subsets of $\X$, and let $\widetilde{h}$ denote $h |_{A \times B}$, the restriction of $h$ to $A \times B$. Then $\widetilde{h}$ is optimal among densities in $\Gamma (\widetilde{h}_X, \widetilde{h}_Y)^{\overline{h}}$ with respect to the same cost $c$.
\end{lem}

\begin{proof}
If $\widetilde{h}$ is not optimal, then there exists a plan $\widetilde{h}'\in \Gamma (\widetilde{h}_X, \widetilde{h}_Y)^{\overline{h}}$ improving $\widetilde{h}$. Note that $\widetilde{h}$ and $\widetilde{h}'$ are {\it both supported on the rectangular neighbourhood} $A \times B$. Now consider the plan $h - \widetilde{h} + \widetilde{h}'$ which improves $h$. Since $\widetilde{h}$ and $\widetilde{h}'$ have the same marginals, $h - \widetilde{h} + \widetilde{h}' \in \Gamma (h_X, h_Y)$. Note that
$$%\begin{eqnarray*}
h - \widetilde{h} + \widetilde{h}' = \left\{
\begin{array}{cc}
\widetilde{h}' &  \mbox{ on } A \times B \\
h &  \mbox{  otherwise},  \end{array}\right.
$$%\end{eqnarray*}
and that $h - \widetilde{h} + \widetilde{h}' \leq \overline{h}$. It follows that the improved plan $h - \widetilde{h} + \widetilde{h}' \in \Gamma (h_X, h_Y)^{\overline{h}}$, contradicting optimality of $h$.
\end{proof}

\begin{rmk}\label{more_examples}
By the above lemma, restricting the optimal density of example~\ref{example} to rectangular sets gives more examples of optimal densities.
\end{rmk}

\begin{propn}\label{propn:third}
Let the cost $c(x,y)$ satisfy conditions $(C1)-(C3)$ of subsection~\ref{assumptions:cost}.
Let $0 \leq \overline{h}\in L^{\infty}(\X \times \Y)$ have compact support and suppose that $\Gamma(f,g)^{\overline{h}}\neq \emptyset$. Make a linear change of coordinates if necessary so that (\ref{Taylor_simplified}) holds.
Take $h\in\Gamma(f,g)^{\overline{h}}$ and let $(x_0,y_0)\in \X\times \Y \setminus Z$ be a Lebesgue point of $h$. Consider the blow-up sequence $h_n$ of $h$ at $(x_0,y_0)$. Then $h$ $c$-optimal implies that $h_n$ is $\widetilde{c}_n$-optimal,
where $\widetilde{c}_n(x,y) = \widetilde c(x,y) + \widetilde R_2(\frac{x}{n},\frac{y}{n})$:
\begin{eqnarray*}
h \in \underset{k\in\Gamma(h_X,h_Y)^{\overline{h}}}{\mbox{\textnormal{argmin}}} I_c(k) \Longrightarrow h_n \in \underset{ k\in\Gamma((h_n)_X,(h_n)_Y)^{\overline{h}_n}}{\mbox{\textnormal{argmin}}} I_{\widetilde{c}_n}(k).
\end{eqnarray*}
\end{propn}

\begin{proof} Let $Q_n=Q_n(x_0,y_0)$ and
consider the blow-up process as being done in two steps: restriction ($h'_n:=h|_{Q_n}$) and dilation ($h_n:=h'_n \circ \varphi_n$).
In the restriction stage $h$ is restricted to the rectangular neighbourhood $Q_n$, hence by
Lemma~\ref{lem:restriction} $h'_n$ is optimal:
\begin{eqnarray*}
h \in \underset{ {k\in\Gamma(h_X,h_Y)^{\overline{h}}}}{\mbox{argmin}} \int_{\X \times \Y} c(x,y)k(x,y) \Longrightarrow \\
{h}'_n \in \underset{k\in\Gamma((h'_n)_X,(h'_n)_Y)^{\overline{h}'_n}}{\mbox{argmin}} \int_{Q_n} c(x,y)k(x,y).
\end{eqnarray*}
In the dilation stage $h'_n$ is composed with the linear map $\varphi_n:Q\rightarrow Q_n: (x,y) \mapsto (\frac{1}{n}x+x_0,\frac{1}{n}y+y_0)$. Note that $det D\varphi_n(x,y)=\frac{1}{n^{2d}}$.
By the change of variables formula,
\begin{eqnarray*}
& & \int_{Q_n} c(x,y)h'_n(x,y) = \int_{Q_n} c(x,y) (h_n \circ \varphi_n^{-1})(x,y)\\
&=&\int_Q c(\varphi_n(x,y)) h_n(x,y)|det D\varphi_n(x,y)| \\
&=& \frac{1}{n^{2d}} \int_Q c(x_0+\frac{x}{n},y_0+\frac{y}{n})h_n(x,y),
\end{eqnarray*}
and so,
\begin{eqnarray*}
& &
\underset{ k\in\Gamma((h'_n)_X, (h'_n)_Y)^{\overline{h}'_n}}{\mbox{\textnormal{argmin}}} \int_{Q_n} c(x,y)k(x,y)\\
&=& \underset{ k\in\Gamma((h_n)_X,(h_n)_Y)^{\overline{h}_n} }{\mbox{\textnormal{argmin}}}  \frac{1}{n^{2d}} \int_Q c(x_0+\frac{x}{n},y_0+\frac{y}{n})k(x,y)\\
%&=& \inf_{k\in\Gamma((h_n)_X,(h_n)_Y)^{\overline{h}_n} } \frac{1}{n^{2(d+1)}} \int_Q n^2 c(x_0+\frac{x}{n},y_0+\frac{y}{n})k(x,y)\\
&=& \underset{k\in\Gamma((h_n)_X,(h_n)_Y)^{\overline{h}_n}}{\mbox{\textnormal{argmin}}} \frac{1}{n^{2(d+1)}} \int_Q \{\widetilde{c}(x, y)+ \widetilde{R}_2(\frac{x}{n},\frac{y}{n})\} k(x,y)\\
&=& \underset{k\in\Gamma((h_n)_X,(h_n)_Y)^{\overline{h}_n}}{\mbox{argmin}} \int_Q \widetilde{c}_n(x,y) k(x,y).
\end{eqnarray*}

For the second equality above, note that those terms of the Taylor expansion (\ref{Taylor_simplified}) which are constant, are functions of $x$ alone, or are functions of $y$ alone give the same value when integrated against any density $k$ in $\Gamma((h_n)_X,(h_n)_Y)$ since the marginals are fixed. Hence for the variational problem at hand only the mixed {$2^{nd}$}-order terms in the Taylor series, namely $\widetilde{c}(x, y)$, and the remainder, $R_2(\frac{x}{n},\frac{y}{n})$, matter. For the last equality above, recall that $\mbox{argmin}\int \cdot=\mbox{argmin} \frac{1}{m}\int\cdot$ for any positive constant $m$.

\end{proof}

\section{Is Optimality inherited by blow-ups?}

It is natural to ask whether the blow-up $h_\infty$ of an optimal $h$ is also optimal (among densities which share its marginals and which are dominated by the blow-up $\overline{h}_\infty$ of $\overline{h}$).
For our purposes we do not need to have a complete answer to this question. Instead, we derive a necessary condition for $h_\infty$ to be (almost) optimal. In section~\ref{section:implies} we show this condition is satisfied when $h$ is optimal.

\begin{lem}\label{lem:necc_cond}
Let the cost $c(x,y)$ satisfy conditions $(C2)-(C3)$ of subsection~\ref{assumptions:cost}.
Let $0 \leq \overline{h}\in L^{\infty}(\X \times \Y)$ have compact support and suppose that $\Gamma(\mu,\nu)^{\overline{h}}\neq \emptyset$.
Take $h\in\Gamma(f,g)^{\overline{h}}$ and let $(x_0,y_0)\in \X\times \Y \setminus Z$ be a common Lebesgue point of $h$ and $\overline{h}$. Let $h_{\infty}, \overline{h}_{\infty}\in L^1(Q)$  be the blow-ups of $h, \overline{h}$ at $(x_0, y_0)$. If $0 < h(x_0,y_0) < \overline{h}
(x_0,y_0)$ then $h_\infty\in \Gamma(f_\infty,g_\infty)^{\overline{h}_\infty}$ can be improved: for any $\delta$ which satisfies $0<\delta <\mbox{min}(h(x_0, y_0),\overline{h}(x_0, y_0)-h(x_0, y_0))$ there exists $h_\infty^\delta \in \Gamma(f_\infty,g_\infty)^{\overline{h}_\infty}$ such that $I_{\widetilde{c}}(h_\infty^\delta) < I_{\widetilde{c}}(h_\infty)$. Furthermore, $h(x_0, y_0)-\delta \leq h_\infty^\delta \leq h(x_0, y_0)+\delta$ on $Q$.
\end{lem}

\begin{proof}
Suppose that $0<h(x_0,y_0)<\overline{h}(x_0,y_0)$ at $(x_0,y_0)$. By Corollary~\ref{cor:blow-up}, $h_\infty$ is equal to the constant function $r=h(x_0,y_0)$ almost everywhere on $Q$. Its marginals $f_\infty=(h_\infty)_X$ and $g_\infty=(h_\infty)_Y$ are both equal to $r$ almost everywhere on $[-\frac{1}{2},\frac{1}{2}]^d$. Also by Corollary~\ref{cor:blow-up}, $\overline{h}_\infty$ is equal to the constant function $R=\overline{h}(x_0,y_0)$ almost everywhere on $Q$. By our assumption $0 < r < R$.

We next recall a standard perturbation argument (e.g.~\cite[proof of Theorem 2.3]{GM}) to show that $r$ is not optimal among densities $k\in\Gamma(r,r)$ constrained by $R$, where optimality is measured against $\widetilde{c}(x,y)= \, x\cdot y$. Pick two points $(x_1,y_1)$ and $(x_2,y_2)$ in $Q$ such that $\widetilde{c}(x_1,y_1) + \widetilde{c}(x_2,y_2) < \widetilde{c}(x_1,y_2) + \widetilde{c}(x_2,y_1)$. Since $\widetilde{c}(x,y)$ is continuous, there exist (compact) neighbourhoods $U_j\subset [-\frac{1}{2},\frac{1}{2}]^d$ of $x_j$ and $V_j\subset [-\frac{1}{2},\frac{1}{2}]^d$ of $y_j$ such that $\widetilde{c}(u_1,v_1) + \widetilde{c}(u_2,v_2) < \widetilde{c}(u_1,v_2) + \widetilde{c}(u_2,v_1)$ whenever $u_j\in U_j$ and $v_j\in V_j$. It follows that $U_1 \cap U_2 \neq \emptyset$ and $V_1 \cap V_2 \neq \emptyset$. Take $0<\delta <\mbox{min}(r,R-r)$ and consider the density $\Delta h$ which is equal $\delta$ on $U_1\times V_1$, $U_2\times V_2$, is $-\delta$ on $U_1\times V_2$, $U_2\times V_1$ and is $0$ everywhere else. Note that $h_\infty=r$ and $h_\infty^\delta:= r+\Delta h$ have the same marginals,
and that $0< r-\delta \leq h_\infty^\delta \leq r + \delta < R$ by choice of $\delta$,
so $h_\infty^\delta\in \Gamma(f_\infty,g_\infty)^{\overline{h}_\infty}$.
By the choice of the points $(x_1,y_1)$ and $(x_2,y_2)$, $h_\infty^\delta=r+\Delta h$ has
lower cost than $h_\infty=r$: $I_{\widetilde{c}}(h_\infty^\delta) < I_{\widetilde{c}}(h_\infty)$.
\end{proof}

\begin{defn}\label{defn:geometrically_extreme} (Geometrically Extreme.) Let $\overline{h}$ be bounded.
A density $h$ in $\Gamma (f, g)^{\overline{h}}$ will be called {\it geometrically extreme} if there exists a ($\Leb^{2d}$-measurable) set $W\subset \mathbb{R}^{2d}$ such that $h(x,y)=\overline{h}(x,y)1_{W}(x,y)$ for almost every $(x,y)\in \mathbb{R}^{2d}$. Here  $1_{W}$ is the characteristic function of the set $W$.
\end{defn}

\begin{cor}\label{cor:geom_ext} (A necessary condition for optimality of $h_\infty$.)
Let the cost $c(x,y)$ satisfy conditions $(C1)-(C3)$ of subsection~\ref{assumptions:cost}.
Let $0 \leq \overline{h}\in L^{\infty}(\X \times \Y)$ have compact support and assume that $\Gamma(f,g)^{\overline{h}}\neq \emptyset$.
Take $h\in\Gamma(f,g)^{\overline{h}}$.
If $h_\infty$ is $\widetilde{c}$-optimal at almost every $(x_0,y_0)$, i.e. $ h_\infty \in \mbox{argmin}_{k\in\Gamma(f_\infty,g_\infty)^{\overline{h}_\infty}} I_{\widetilde{c}}(k)$, then $h$ is geometrically extreme.
\end{cor}

\begin{proof}
Let $N:=\{(x, y)\in\X\times\Y\setminus Z\, |\, (x, y)\mbox{ is a common Lebesgue}$
$\mbox{point of }h \mbox{ and } \overline{h}  \}$, where $Z$ is the Lebesgue negligible set of subsection~\ref{assumptions:cost}. Recall that almost every point in $\X \times \Y$ is in $N$.
Being $\widetilde{c}$-optimal, $h_\infty$ cannot be improved. Hence by Lemma~\ref{lem:necc_cond}, $h(x_0,y_0)$ is either equal to $0$ or equal to $\overline{h}(x_0,y_0)$ at each point $(x_0,y_0)\in N$. In other words, $h$ is geometrically extreme.
\end{proof}

\section{Optimality implies being geometrically extreme}\label{section:implies}

The following lemma will be used in the proof of Theorem~\ref{thm:optimality_ge}.
Given two {\it not necessarily positive} marginal densities
$f, g\in L^1$ with the same total mass $\int f = \int g$, we would like to produce a joint density $h$ which is controlled by $f$ and $g$. Since $f$ and $g$ are not necessarily positive, it is possible for their total mass to be zero even when the densities themselves are not identically zero. In such a case the product $f(x)g(y)$ does not necessarily have $f$ and $g$ as its marginals. The following lemma addresses this issue.\\

Let $\phi[Z]:=\int_Z \phi(z)dz$ denote the total mass of the function $\phi$ on the set $Z$.

\begin{lem}\label{lemma:first}
Let $X, Y$ be Borel subsets of $[-\frac{1}{2},\frac{1}{2}]^d$ whose $\Leb^d$-measure is strictly positive. Let $f\in L^1(X)$, $g\in L^1(Y)$ have same total mass $m:=f[X]= g[Y] \in \mathbb{R}$. Suppose $||f||_{L^\infty(X)}, ||g||_{L^\infty(Y)} < \epsilon$. Then there exists a joint density $h\in L^1(X \times Y)$ with marginals $f$ and $g$ such that $ ||h||_{L^\infty(X\times Y)}< 3\epsilon(\frac{1}{\Leb^d[X]} + \frac{1}{\Leb^d[Y]})$.
\end{lem}

\begin{proof}
Let $f_0:=\frac{1}{\Leb^d[X]}\in L^1(X)$ and $g_0:=\frac{1}{\Leb^d[Y]} \in L^1(Y)$. Note that $f_0$ and $g_0$ have total mass $1$. We first deal with the case $m=0$. Note that $(f\cdot g_0)_X=f$,
%\begin{eqnarray*}
%\pi_x^{\#}(\mu\times\nu_0)[U]&=&(\mu\times\nu_0)[\pi_x^{-1}(U)]=(\mu\times\nu_0)[U \times Y]\\
%&=& \mu[U]\nu_0[Y] =\mu[U] \quad \mbox{ for all Borel } U\subset X,
%\end{eqnarray*}
while $(f \cdot g_0)_Y=0$.
%\begin{eqnarray*}
%\pi_y^{\#}(\mu\times\nu_0)[V]&=&(\mu\times\nu_0)[\pi_y^{-1}(V)]=(\mu\times\nu_0)[X \times V]\\
%&=& \mu[X]\nu_0[V] =0 \quad \mbox{ for all Borel } V\subset Y.
%\end{eqnarray*}
Similarly, $(f_0 \cdot g)_X=0$, while $(f_0 \cdot g)_Y=g$.
Let $h:=f \cdot g_0 + f_0\cdot g$. Since the maps $(\cdot )_X$ and $(\cdot )_Y$ are linear, we get that $h_X=f$, and $h_Y=g$.

More generally, suppose the total mass $m=f[X]=g[Y]$ is not necessarily $0$. Let $h:=f \cdot g_0 + f_0\cdot g - m f_0\cdot g_0=(f - m f_0)\cdot g_0 + f_0\cdot (g -m g_0) + m f_0\cdot g_0$. Since the total mass of $f - m f_0$ and $g -mg_0$ is $0$, we conclude by above that $h_X=(f - m f_0) + 0 + m f_0 = f$, and $h_Y=0 + (g -m g_0)+ m g_0 = g$. For $(x,y)\in X\times Y$ the density $h$ satisfies:
\begin{eqnarray*}
|h(x,y)| &\leq& |(f - m f_0)(x)||g_0(y)|+|f_0(x)||(g - m g_0)(y)|\\
& & +|m||f_0(x)||g_0(y)|\\
&\leq& (|f(x)| + |m||f_0(x)|)|g_0(y)|\\
& & +(|g(y)| + |m||g_0(y)|)|f_0(x)|\\
& & +|m||f_0(x)||g_0(y)|\\
&\leq & \frac{1}{\Leb^d[Y]}(|f(x)|+\frac{|m|}{\Leb^d[X]})\\
& & +\frac{1}{\Leb^d[X]}(|g(y)|+\frac{|m|}{\Leb^d[Y]}) + \frac{|m|}{\Leb^d[X]\Leb^d[Y]}\\
&\leq& \frac{2\epsilon}{\Leb^d[Y]}+\frac{2\epsilon}{\Leb^d[X]}+\frac{\epsilon}{\Leb^d[X]}\\
&<& 3\epsilon(\frac{1}{\Leb^d[X]}+\frac{1}{\Leb^d[Y]}).
\end{eqnarray*}
The penultimate inequality above uses that $|m|=|\int_X f(x)dx|\leq \int_X |f(x)|dx \leq \epsilon \Leb^d[X]$, and $|m|=|\int_Y g(y)dy|\leq \int_Y |g(y)|dy \leq \epsilon \Leb^d[Y]$.
\end{proof}

\begin{thm}\label{thm:optimality_ge}
Let the cost $c(x,y)$ satisfy conditions $(C1)-(C3)$ of subsection~\ref{assumptions:cost}.
Let $0 \leq \overline{h}\in L^\infty (\X \times \Y)$ have compact support and take $0 \leq f, g \in L^1_{c}(\X \times \Y)$ such that $\Gamma(f,g)^{\overline{h}}\neq \emptyset$.
If $h\in\Gamma(f,g)^{\overline{h}}$ is optimal, i.e. $ h \in \mbox{argmin}_{k\in\Gamma(f,g)^{\overline{h}}} I_{c}(k)$, then $h$ is geometrically extreme.
\end{thm}

\begin{proof}
Let $N:=\{(x,y)\in \X \times \Y \setminus Z \,|\, (x,y) \mbox{ is a common Lebesgue}$ $\mbox{point of } h \mbox{ and } \overline{h} \}$ where $Z$ be the Lebesgue negligible set of subsection~\ref{assumptions:cost}. Note that almost every point in $\X \times \Y$ is in $N$.

Fix $(x_0,y_0)\in N$ and let $h_n$ and $\overline{h}_n$ be the blow-up sequences of $h$ and $\overline{h}$ at $(x_0,y_0)$, with $h_\infty$ and $\overline{h}_\infty$ their respective limits in $L^1$. Suppose by contradiction that $0 < h(x_0,y_0) < \overline{h}(x_0,y_0)$.

Let $\overline{R}:=\|\overline{h}\|_{L^\infty(Q_1)} ,R:=\overline{h}(x_0, y_0)$, and $r:= h(x_0, y_0)$.
By Lemma~\ref{lem:necc_cond}, for $0<\delta <\mbox{min}(r,R-r)$, there exists $h_\infty^\delta \in \Gamma(f_\infty,g_\infty)^{\overline{h}_\infty}$ such that $r-\delta \leq h_\infty^\delta \leq r+\delta$ and
\begin{equation}\label{eqn:improved}
I_{\widetilde{c}}(h_\infty^\delta) < I_{\widetilde{c}}(h_\infty).
\end{equation}
Assume for now (argued below) that there exists a sequence of non-negative densities
$h^{\delta}_n\in L^1(Q)$ ($n\in N_0$, where the index set $N_0$ is the set of natural
numbers defined in Propositions~\ref{propn:first}--\ref{propn:second}),
with the following properties for large enough $n$:

\begin{itemize}
\item[(P1)]{$h^{\delta}_n \leq \overline{h}_n$ on $\widetilde{Z}_n^{good}$, where $\widetilde{Z}_n^{good} \subset Q$ is a rectangular set satisfying $\Leb^{2d}[\widetilde{Z}_n^{good}] \rightarrow 1$ as $n \rightarrow \infty$,}
\item[(P2)]{$h^{\delta}_n|_{\widetilde{Z}_n^{good}}$ and $h_n|_{\widetilde{Z}_n^{good}}$ have the same marginals,}
\item[(P3)]{$h^{\delta}_n$ is bounded by a constant $\frac{(\overline{R}+1)^3}{r^2}$ independent of $n$ on $Q\setminus\widetilde{Z}_n^{good}$,}
\item[(P4)]{$I_{\widetilde{c}_n}(h^{\delta}_n)\longrightarrow I_{\widetilde{c}}(h^{\delta}_\infty)$ as $n\rightarrow \infty$.}
\end{itemize}

By Lemma~\ref{lem:restriction} constrained optimality is inherited by restriction to rectangular sets. Hence since, by Proposition~\ref{propn:third}, $h_n$ is $\widetilde{c}_n$-optimal among all densities which share its marginals and which are dominated by $\overline{h}_n$, its restriction, $h_n|_{\widetilde{Z}_n^{good}}$, remains $\widetilde{c}_n$-optimal among all densities which share its marginals and which are dominated by $\overline{h}_n$. In particular, by (P1)--(P2), $I_{\widetilde{c}_n}(h^{\delta}_n|_{\widetilde{Z}_n^{good}}) \geq I_{\widetilde{c}_n}(h_n|_{\widetilde{Z}_n^{good}})$. Hence,
\begin{eqnarray}\label{eqn:asymptote}
I_{\widetilde{c}_n}(h^{\delta}_n)&=&I_{\widetilde{c}_n}(h^{\delta}_n|_{\widetilde{Z}_n^{good}})+I_{\widetilde{c}_n}(h^{\delta}_n|_{Q\setminus\widetilde{Z}_n^{good}})\nonumber\\
&\geq& I_{\widetilde{c}_n}(h_n|_{\widetilde{Z}_n^{good}})+I_{\widetilde{c}_n}(h^{\delta}_n|_{Q\setminus\widetilde{Z}_n^{good}})\nonumber\\
&=& I_{\widetilde{c}}(h_n|_{\widetilde{Z}_n^{good}})+I_{\widetilde{c}}(h^{\delta}_n|_{Q\setminus\widetilde{Z}_n^{good}})+o(1)\\
&=& I_{\widetilde{c}}(h_n)-I_{\widetilde{c}}(h_n|_{Q\setminus\widetilde{Z}_n^{good}}) +I_{\widetilde{c}}(h^{\delta}_n|_{Q\setminus\widetilde{Z}_n^{good}})+o(1)\nonumber,
\end{eqnarray}
where we have used equation (\ref{eqn:o(1)}) to go from the second line to the third.

Note that since $|x \cdot y|\leq d/4$ on $Q$, for any $k\in L^1(Q)$: $|I_{\widetilde{c}}(k)| \leq \int_Q|x \cdot y|k(x,y)\leq \frac{d}{4} k[Q]$. Hence, rearranging equation (\ref{eqn:asymptote}) we get
\begin{eqnarray*}
I_{\widetilde{c}}(h_n) - I_{\widetilde{c}_n}(h^{\delta}_n) + o(1) &\leq & I_{\widetilde{c}}(h_n|_{Q\setminus\widetilde{Z}_n^{good}}) - I_{\widetilde{c}}(h^{\delta}_n|_{Q\setminus\widetilde{Z}_n^{good}})\\
&\leq & |I_{\widetilde{c}}(h_n|_{Q\setminus\widetilde{Z}_n^{good}})| + |I_{\widetilde{c}}(h^{\delta}_n|_{Q\setminus\widetilde{Z}_n^{good}})|\\
&\leq & \frac{d}{4}(h_n[{Q\setminus\widetilde{Z}_n^{good}}]+ h^{\delta}_n[{Q\setminus\widetilde{Z}_n^{good}}])\\
&\leq & \frac{d}{4}(\overline{R} + \frac{(\overline{R}+1)^3}{r^2})\Leb^{2d}[{Q\setminus\widetilde{Z}_n^{good}}],
\end{eqnarray*}
where the last inequality above follows from (c) of Proposition~\ref{propn:first} and property (P3).
Letting $n \rightarrow \infty$ above, and using properties (P1) and (P4) as well as the continuity of the linear functional $I_{\widetilde{c}}(\cdot)$, we get that $I_{\widetilde{c}}(h^{\delta}_\infty) \geq I_{\widetilde{c}}(h_\infty)$, contradicting equation (\ref{eqn:improved}). Hence for every $(x_0,y_0)\in N$ either $0 = h(x_0,y_0)$ or $ h(x_0,y_0) = \overline{h}(x_0,y_0)$. In other words, $h$ is geometrically extreme.\\

In the rest of this proof we demonstrate the existence of a sequence $h^{\delta}_n$ with properties (P1)--(P4). We do this in several steps. For ease of reference, we record the following chain of inequalities when $0<\delta <\mbox{min}(r,R-r)$:
$$0  < \frac{r-\delta}{2} < r-\delta \leq h_\infty^\delta \leq r+\delta < \frac{R+r+\delta}{2} < \frac{3R+r+\delta}{4} < R.$$

Recall we are supposing by contradiction that $0 < h(x_0,y_0) < \overline{h}(x_0,y_0)$.

Step (1): Construction of densities $\widehat{h}^\delta_n$. Let $f_n:=(h_n)_X$, $g_n:=(h_n)_Y$, $f_\infty:=(h_\infty)_X$, and $g_\infty:=(h_\infty)_Y$. Note that $f_n$ and $g_n$ have same total mass, and that $f_\infty$ and $g_\infty$ have same total mass. Since $f_n$ and $f_\infty$ may not have the same total mass, we will work with normalized copies $f'_n:=\frac{h_\infty [Q]}{h_n [Q]}f_n$ and $g'_n:=\frac{h_\infty [Q]}{h_n [Q]}g_n$. It follows from remark~\ref{cor:blow-up} that $h_n [Q] \rightarrow h_\infty [Q] > 0$, hence $f'_n$ and $g'_n$ are well-defined, at least for large enough $n$ which is all we will use. Note that for large enough $n$, $f'_n, g'_n, f_\infty$ and $g_\infty$ all have the same total mass. Since $\overline{h}$ is bounded and of compact support, so is $h$, hence so are $f_n, g_n, f'_n, g'_n \in L^1([-\frac{1}{2},\frac{1}{2}]^d)$, as well as $f_\infty, g_\infty\in L^1([-\frac{1}{2},\frac{1}{2}]^d)$.

Let $S_n:[-\frac{1}{2},\frac{1}{2}]^d \rightarrow [-\frac{1}{2},\frac{1}{2}]^d$ be the unique measure preserving map between $f_\infty$ and $f'_n$ (see~\cite{GM95}) minimizing the cost $\widetilde{c}(x,y)= x \cdot y$. Similarly let  $T_n:[-\frac{1}{2},\frac{1}{2}]^d \rightarrow [-\frac{1}{2},\frac{1}{2}]^d$ be the unique measure preserving map between $g_\infty$ and $g'_n$ minimizing $\widetilde{c}(x,y)$ . Note that $S_n$ and $T_n$ are essentially {\it bijections} (see~\cite{GM}).

Recall (e.g.~\cite{GM95}) that a measure preserving map $s$ between two $L^1$-functions $f$ and $g$ is a Borel map which satisfies the change of variables formula
\begin{equation}\label{eqn_change_of_variables}
\int_{\Y} h(y)g(y)dy = \int_{\X} h(s(x))f(x)dx,
\end{equation}
for all $h$ continuous on $\Y$. Given $f\in L^1(\X)$ and a Borel map $s:\X \rightarrow \Y$, there is a unique function $g \in L^1(\Y)$ satisfying equation (\ref{eqn_change_of_variables}). Call this $g$ the {\it push forward} of $f$ by $s$, denoted $s_\#f$. Note that $s$ is measure preserving between $f$ and $s_\#f$. Whenever $s$ is a diffeomorphism, equation (\ref{eqn_change_of_variables}) implies
\begin{equation}\label{eqn_local_change_of_variables}
g(s(x))| det Ds(x) |=f(x).
\end{equation}
From~\cite{M97}, if $s$ fails to be a diffeomorphism but is given by
the gradient of a convex function,
equation (\ref{eqn_local_change_of_variables}) continues to hold $f$-a.e.

Recall ~\cite{ B91, M95} that for the cost $\widetilde{c}(x,y)= x \cdot y$, the optimal maps $S_n(x)$ and $T_n(y)$ have the form $x \mapsto \nabla \psi (x)$ and $y \mapsto \nabla \phi (y)$, where $\psi$ and $\phi$ are convex functions. By Alexandrov's Theorem a convex function has second order derivatives almost everywhere. Hence it makes sense to talk about the derivatives $DS_n$ and $DT_n$ almost everywhere. \\

We note that $(S_n \times T_n)_\#k\in\Gamma(f'_n,g'_n)$ for any $k\in \Gamma(f_\infty,g_\infty)$. It is straightforward to see this: we check that $((S_n \times T_n)_\# k)_X = f'_n$ (checking $((S_n \times T_n)_\# k)_Y = g'_n$ is similar). For any $h\in C([-\frac{1}{2},\frac{1}{2}]^d)$:
\begin{eqnarray*}\label{eqn:rightmarginals}
& & \int h(x) ((S_n\times T_n)_\# k)_X (x) = \int \int h(x) ((S_n\times T_n)_\# k)(x,y) \\
&=& \int \int h(S_n(x)) k(x,y) = \int h(S_n(x)) f_\infty(x) = \int h(x) f'_n(x).
\end{eqnarray*}

Let $\widehat{h}_n^{\delta}:=\frac{h_n [Q]}{h_\infty [Q]}(S_n \times T_n)_{\#} h^{\delta}_\infty$. By the above, $\widehat{h}_n^{\delta}\in\Gamma(f_n,g_n)$ for all $n$, that is $\widehat{h}_n^{\delta}$ has the same marginals as $h_n$.\\

Step (2): We next show that $\widehat{h}_n^{\delta}$, where $n\in N_0$, satisfies property (P1).

Recall the notation of Proposition~\ref{propn:second}.
Denoting $\overline{X}_n^{good}:=S^{-1}_n(\widetilde{X}_n^{good})$ (respectively $\overline{Y}_n^{good}:=T^{-1}_n(\widetilde{Y}_n^{good})$), we have that $\Leb^{d}[\overline{X}_n^{good}] \rightarrow 1$ (respectively that $\Leb^{d}[\overline{Y}_n^{good}] \rightarrow 1)$.

By (b) of Proposition~\ref{propn:second}, $f_n \rightarrow h(x_0,y_0)=r$ `uniformly' on $\widetilde{X}_n^{good}$. By (c) of Proposition~\ref{propn:second}, $f_n\leq \|\overline{h}\|_{L^\infty(Q_1)}=\overline{R}$ on $[-\frac{1}{2}, \frac{1}{2}]^d$.
Since $S_n$ is a convex gradient, equation (\ref{eqn_local_change_of_variables}) applies to give,
$$\frac{1}{|det(DS_n)(x)|}=\frac{f'_n(S_n(x))}{f_\infty(x)}=\frac{h_\infty [Q]}{h_n [Q]} \frac{f_n(S_n(x))}{r}\longrightarrow 1$$ `uniformly' on $\overline{X}_n^{good}$;
while on $[-\frac{1}{2}, \frac{1}{2}]^d$,  $\frac{1}{|det(DS_n)|} \leq \frac{\overline{R}}{r}\frac{h_\infty [Q]}{h_n [Q]} < \frac{\overline{R}+1}{r}$ for large enough $n$.

Similarly, by Proposition~\ref{propn:second} and equation (\ref{eqn_local_change_of_variables}),
$$\frac{1}{|det(DT_n)(y)|}=\frac{g'_n(T_n(y))}{g_\infty(y)}=\frac{h_\infty [Q]}{h_n [Q]}\frac{g_n(T_n(y))}{r}\longrightarrow 1$$ `uniformly' on $\overline{Y}_n^{good}$;
while on $[-\frac{1}{2}, \frac{1}{2}]^d$,  $\frac{1}{|det(DT_n)|} \leq \frac{\overline{R}}{r}\frac{h_\infty [Q]}{h_n [Q]}< \frac{\overline{R}+1}{r}$ for large enough $n$.

Hence,
$$\frac{1}{|det(D(S_n \times T_n))(x,y)|}
= \frac{1}{|det(DS_n)(x)||det(DT_n)(y)|} \longrightarrow 1
$$
`uniformly' on $\overline{Z}_n^{good}:=\overline{X}_n^{good} \times \overline{Y}_n^{good}$; while on $Q$, $\frac{1}{|det(D(S_n \times T_n))|} < (\frac{\overline{R}+1}{r})^2$ for large enough $n$.

Note that the optimal map from $(f_\infty, g_\infty)$ to $(f'_n, g'_n)$ is given by $(x,y) \rightarrow (S_n(x), T_n(y))= \nabla (\psi(x) + \phi (y))$, a gradient of a convex function. So equation (\ref{eqn_local_change_of_variables}) applies to give, $\widehat{h}_n^\delta((S_n \times T_n)(x,y))\frac{h_\infty [Q]}{h_n [Q]}=\frac{h_\infty^\delta(x,y)}{|det(D(S_n \times T_n))(x,y)|}$. It follows that for large enough $n$,
\begin{equation}\label{eqn:perturbed_gamma}
0 < \frac{r -\delta}{2} < \widehat{h}_n^\delta((S_n \times T_n)(x,y)) < \frac{R+r+\delta}{2} < R,
\end{equation}
for almost every $(x,y)\in \overline{Z}_n^{good}$; while on $Q$ and for large enough $n$,
\begin{equation}\label{eqn:widehat_gamma_n_delta}
\widehat{h}_n^\delta < (\frac{\overline{R}+1}{r})^2 R \leq \frac{(\overline{R}+1)^3}{r^2}.
\end{equation}

Recall that by (b) of Corollary~\ref{cor:blow-up}, $\overline{h}_n \rightarrow R$ uniformly on $Q$. It follows, using equation (\ref{eqn:perturbed_gamma}), that for large enough $n$, $\widehat{h}_n^\delta|_{\widetilde{Z}_n^{good}} \leq \overline{h}_n|_{\widetilde{Z}_n^{good}}$.\\

Step (3): Note that even though $\widehat{h}_n^\delta$ and $h_n$ have the same marginals on $Q$, the marginals of $\widehat{h}_n^\delta|_{\widetilde{Z}_n^{good}}$ and $h_n|_{\widetilde{Z}_n^{good}}$ may not be the same. In step (4) $\widehat{h}_n^{\delta}$ will be perturbed by a density $\widetilde{h}_n^{\delta}$ so that $(\widehat{h}_n^{\delta}+\widetilde{h}_n^{\delta})|_{\widetilde{Z}_n^{good}}$ and $h_n|_{\widetilde{Z}_n^{good}}$ have the same marginals. The perturbation will be chosen to preserve the capacity bound on $\widetilde{Z}_n^{good}$: $(\widehat{h}_n^{\delta}+\widetilde{h}_n^{\delta})|_{\widetilde{Z}_n^{good}} \leq \overline{h}_n|_{\widetilde{Z}_n^{good}}$. In this step we construct $\widetilde{h}_n^{\delta}$.

Let $\widetilde{f}_n^{\delta}=((h_n-\widehat{h}_n^{\delta})|_{\widetilde{Z}_n^{good}})_X \in L^1(\widetilde{X}_n^{good})$ and $\widetilde{g}_n^{\delta}=((h_n-\widehat{h}_n^{\delta})|_{\widetilde{Z}_n^{good}})|_Y\in L^1(\widetilde{Y}_n^{good})$ be the marginals of $(h_n-\widehat{h}_n^{\delta})|_{\widetilde{Z}_n^{good}}$.
Since $h_n$ and $\widehat{h}_n^{\delta}$ have the same marginals on $Q$, $\int_{\widetilde{Y}_n^{good}}(h_n- \widehat{h}_n^{\delta})(x,y)dy + \int_{\widetilde{Y}_n^{bad}} (h_n- \widehat{h}_n^{\delta})(x,y)dy = 0$.

Hence, by (c) of Proposition~\ref{propn:first} and equation (\ref{eqn:widehat_gamma_n_delta}),
\begin{eqnarray*}
|\widetilde{f}_n^{\delta}(x)|
& = & |\int_{\widetilde{Y}_n^{good}} (h_n-\widehat{h}_n^{\delta})(x,y)dy|
= |\int_{\widetilde{Y}_n^{bad}} (h_n-\widehat{h}_n^{\delta})(x,y)dy| \\
& \leq & \int_{\widetilde{Y}_n^{bad}} |h_n-\widehat{h}_n^{\delta}|(x,y)dy
 \leq \int_{\widetilde{Y}_n^{bad}}(|h_n|+|\widehat{h}_n^{\delta}|)(x,y)dy\\
& \leq & (\overline{R} + \frac{(\overline{R}+1)^3}{r^2})  \Leb^{d} [ \widetilde{Y}_n^{bad} ].
\end{eqnarray*}
Similarly $|\widetilde{g}_n^{\delta}(y)|\leq (\overline{R} + \frac{(\overline{R}+1)^3}{r^2}) \Leb^{d} [ \widetilde{X}_n^{bad} ]$.
It follows from Lemma~\ref{lemma:first}, that there exist a joint density $\widetilde{h}_n^{\delta}\in \mathcal{M}(\widetilde{X}_n^{good}\times \widetilde{Y}_n^{good})$  with marginals
$\widetilde{f}_n^{\delta}$ and $\widetilde{g}_n^{\delta}$
such that
\begin{equation}\label{eqn:bound_on_perturbation}
|\widetilde{h}_n^{\delta}(x,y)|< 3(\overline{R} + \frac{(\overline{R}+1)^3}{r^2})(\Leb^{d} [ \widetilde{X}_n^{bad} ]+\Leb^{d} [ \widetilde{Y}_n^{bad} ])(\frac{1}{\Leb^{d} [ \widetilde{X}_n^{good} ]}+\frac{1}{\Leb^{d} [ \widetilde{Y}_n^{good} ]}).
\end{equation}
Since the right hand side of equation (\ref{eqn:bound_on_perturbation}) tends to $0$ as $n \rightarrow \infty$, by choosing $n$ large enough, we can make sure the densities $\widetilde{h}_n^{\delta}$ are as close to $0$ as we like. In particular, for large enough $n$,
\begin{equation}\label{eqn:density_of_perturbation}
|(\widetilde{h}_n^{\delta})(x,y)| < \mbox{min}\{ \frac{R-(r+\delta)}{4}, \frac{r-\delta}{4} \}. \end{equation}

Step (4): Establishing properties (P1)--(P4) for the densities $h^\delta_n$.

Let $h_n^{\delta}:=\widehat{h}_n^{\delta}+\widetilde{h}_n^{\delta}$.
Note that although $\widetilde{h}_n^{\delta}$ could be negative, ${h}_n^{\delta}$ is non-negative: from equations (\ref{eqn:perturbed_gamma}) and (\ref{eqn:density_of_perturbation}) we have that $h_n^{\delta}=\widehat{h}_n^{\delta}+\widetilde{h}_n^{\delta} > \frac{r-\delta}{2} - \frac{r-\delta}{4}=\frac{r-\delta}{4} >0$.
Since the marginals of $\widetilde{h}_n^{\delta}$ are $\widetilde{f}_n^{\delta}=((h_n-\widehat{h}_n^{\delta})|_{\widetilde{Z}_n^{good}})|_X$ and $\widetilde{g}_n^{\delta}=((h_n-\widehat{h}_n^{\delta})|_{\widetilde{Z}_n^{good}})|_Y$, $h_n^{\delta}$ and $h_n$ have the same marginals on $\widetilde{Z}_n^{good}$. This establishes property (P2).\\

By (b) of Remark \ref{cor:blow-up}, $\overline{h}_n\rightarrow R$ uniformly on $Q$. Hence, since $R > \frac{3R+r+\delta}{4}$, for large enough $n$: $\overline{h}_n > \frac{3R+r+\delta}{4}$ on $Q$. On the other hand by equations (\ref{eqn:perturbed_gamma}) and  (\ref{eqn:density_of_perturbation}), for large enough $n$, $h_n^{\delta}=\widehat{h}_n^{\delta}+\widetilde{h}_n^{\delta} < \frac{R+r+\delta}{2}+ \frac{R-(r+\delta)}{4}=\frac{3R+r+\delta}{4} < \overline{h}_n$
on $\widetilde{Z}_n^{good}$. This establishes property (P1).\\

Since the perturbation $\widetilde{h}_n^{\delta}$ is supported on $\widetilde{Z}_n^{good}$, $h_n^{\delta}=\widehat{h}_n^{\delta}$ on $Q\setminus\widetilde{Z}_n^{good}$. Hence, using equation (\ref{eqn:widehat_gamma_n_delta}),
$h_n^{\delta}=\widehat{h}_n^{\delta} < \frac{(\overline{R}+1)^3}{r^2}$ on $Q\setminus\widetilde{Z}_n^{good}$. This established property (P3).\\

To establish property (P4) we need to show that $I_{\widetilde{c}_n}(h_n^{\delta})=I_{\widetilde{c}_n}(\widehat{h}_n^{\delta}+\widetilde{h}_n^{\delta}) \rightarrow I_{\widetilde{c}}(h_{\infty}^{\delta})$ as $n \rightarrow \infty$. Note that by equation (\ref{eqn:bound_on_perturbation}) $\widetilde{h}_n^{\delta}\rightarrow 0$ uniformly on $Q$. Hence by equation (\ref{eqn:o(1)}) and Lebesgue's Dominated Convergence Theorem, $I_{\widetilde{c}_n}(\widetilde{h}_n^{\delta})=I_{\widetilde{c}}(\widetilde{h}_n^{\delta})+o(1)\rightarrow 0$ as $n\rightarrow \infty$. So we need only show $I_{\widetilde{c}_n}(\widehat{h}_n^{\delta}) \rightarrow I_{\widetilde{c}}(h_{\infty}^{\delta})$ as $n \rightarrow \infty$.

Stability of the transport map~\cite[Corollary 5.23]{Villani2} implies that $S_n, T_n$ converge in measure to $-id|_{[-\frac{1}{2},\frac{1}{2}]^d}$, minus the identity map on $[-\frac{1}{2},\frac{1}{2}]^d$. By extracting a subsequence if necessary we can assume~\cite[Proposition 4.17]{Royden} that  $S_n, T_n$ converge to $-id|_{[-\frac{1}{2},\frac{1}{2}]^d}$ almost everywhere on $[-\frac{1}{2},\frac{1}{2}]^d$. Since $\widetilde{c}(\cdot, \cdot)$ is continuous, it follows that $\widetilde{c}(S_n(x),T_n(y))$ converges to $\widetilde{c}(x,y)$ almost everywhere on $Q=[-\frac{1}{2},\frac{1}{2}]^d \times [-\frac{1}{2},\frac{1}{2}]^d$.

Note that $|\widetilde{c}(S_n(x),T_n(y))h_{\infty}^{\delta}(x,y)|$ is bounded above on $Q$, e.g. by $\|\widetilde{c}\|_{L^\infty(Q)}R$. Hence, since $\Leb^{2d}[Q]<\infty$, we can apply the Dominated Convergence Theorem to concluded that as $n \rightarrow \infty$:
\begin{eqnarray*}
I_{\widetilde{c}_n}(\widehat{h}_n^{\delta}) &=& \int_Q \widetilde{c}(x,y) \widehat{h}_{n}^{\delta}(x,y)+ o(1)\\
&=& \frac{h_n [Q]}{h_\infty [Q]} \int_Q \widetilde{c}(S_n(x),T_n(y)) h_{\infty}^{\delta}(x,y) + o(1)\\
&\rightarrow & \int_Q \widetilde{c}(x,y) h_{\infty}^{\delta}(x,y)
= I_{\widetilde{c}}(h_{\infty}^{\delta}).
\end{eqnarray*}

This established property (P4) and completes the proof.
\end{proof}

\section{Optimal solution to the constrained problem is unique}\label{section:uniqueness}

We now show that, given a capacity constraint $\overline{h}$, the corresponding constrained optimization problem has a unique solution. In the unconstrained optimization setup, a characteristic property of optimal solutions is {\it c-cyclical monotonicity}. This property can be used to prove a solution is unique~\cite[Theorem 3.7]{GM}. The property of optimal solutions in the constrained setup that is used here to prove uniqueness is that of being geometrically extreme (see Definition~\ref{defn:geometrically_extreme}). Note that in the unconstrained case, c-cyclical monotonicity is in fact necessary and sufficient for optimality,
whereas in the constrained case geometric extremality is merely necessary.

\begin{thm}\label{thm:uniqueness} (Uniqueness)
Let the cost $c(x,y)$ satisfy conditions $(C1)-(C3)$ of subsection~\ref{assumptions:cost}. Let the capacity bound $0\leq \overline{h} \in L^\infty (\X \times \Y)$ have compact support. Take $0 \leq f, g \in L^1_{c}( \X \times \Y)$ such that $\Gamma(f, g)^{\overline{h}}\neq \emptyset$. Then an optimal solution to the constrained problem (\ref{eqn: constraind_optimal_cost}) is unique (as an element of $L^1( \X \times \Y)$).
\end{thm}

\begin{proof}
Suppose $h_1, h_2$ are two optimal plans: $h_1, h_2 \in\underset{k\in\Gamma(f,g)^{\overline{h}}}{\mbox{argmin}} I_c(k)$. We show $h_1 = h_2$ almost everywhere.
Since $\Gamma(f, g)^{\overline{h}}$ is convex, $\frac{1}{2}h_1+\frac{1}{2}h_2 \in \Gamma(f, g)^{\overline{h}}$. Since $I_c(\cdot)$ is linear, the plan $\frac{1}{2}h_1+\frac{1}{2}h_2$ is also optimal.

Hence, by Theorem~\ref{thm:optimality_ge}, $h_1, h_2, \frac{1}{2}h_1+\frac{1}{2}h_2$ are all geometrically extreme. In particular,
$h_i= \overline{h} 1_{W_i}$ almost everywhere on $\X \times \Y$ for $i=1, 2$. Let $\Delta:=(W_1 \setminus W_2) \cup (W_2 \setminus W_1)$ be the symmetric difference of the sets $W_1$ and $W_2$. Either $\Leb^{2d}(\Delta)=0$, in which case $h_1 = h_2$ almost everywhere, or else for almost every $(x,y)\in \Delta$, $0 < (\frac{1}{2}h_1+\frac{1}{2}h_2)(x,y) < \overline{h}(x,y)$, contradicting Theorem~\ref{thm:optimality_ge}.
\end{proof}

\section{Appendix: Duality and Examples}

In this appendix we sketch how the analog of Kantorovich duality \cite{Kantorovich42}
would look for the constrained problem,  following the minimax heuristics
in \cite{AmbrosioGigli11p} \cite{MG}.
One of the virtues of such a duality is that it
makes it easy to check whether a conjectured optimizer is actually optimal.
Defering the elaboration of a full duality theory to a future manuscript~\cite{KM12},
below we develop just enough theory to confirm the claims made in example~\ref{example}.

Suppose $f$ and $g$ have total mass $1$ on $\X$ and recall the Duality Theorem from linear programming (e.g.~\cite{Villani1}). In the unconstrained context the primal problem is (\ref{unconstrained: primal}) and the dual problem is
\begin{equation}
\underset{(u,v)\in Lip_c}{\mbox{sup}} -\int_\X u(x)f(x)dx-\int_\Y v(y)g(y)dy,
\end{equation}
where
$Lip_c:=\{(u,v)\in L^1(\X)\times L^1(\Y)\,|\, c(x,y)+u(x)+v(y) \geq 0 \mbox{ for all } (x,y)\in\X\times \Y\}$. We now formulate a dual problem in the constrained context.
%when the density of $\overline{\gamma}$ is strictly positive: $\overline{h}>0$.
For the primal problem (\ref{eqn: constraind_optimal_cost}) we consider the following dual problem
\begin{equation}\label{unconstrained:dual}
\underset{(u,v,w)\in \overline{Lip}_c}{\mbox{sup}} -\int_\X u(x)f(x)dx-
\int_\Y v(y)g(y)dy+\int_{\X \times \Y} w(x,y) \overline{h}(x,y)dxdy,
\end{equation}
where
$\overline{Lip}_c:=\{(u,v,w)\in L^1(\X)\times L^1(\Y)\times L^1(\X \times \Y)\,|\, c(x,y)+u(x)+v(y)-w(x,y) \geq 0
\mbox{ and } w(x,y) \leq 0 \mbox{ for all } (x,y)\in \X\times \Y\}$.
It follows from the definition of $(u,v,w)\in \overline{Lip}_c$, by integrating $c(x,y) \geq -u(x)-v(y)+w(x,y) $ against $h\in\Gamma(f,g)^{\overline{h}}$, that
\begin{eqnarray*}
& &\int_{\X\times\Y} c(x,y)h (x,y) \geq \int_{\X\times\Y} \{-u(x)-v(y)+w(x,y)\} h (x,y) \geq \\ & & -\int_\X u(x)f(x)-\int_\Y v(y)g(y)+\int_{\X\times\Y} w(x,y)\overline{h}(x,y).
\end{eqnarray*}
Hence when
\begin{eqnarray}\label{eqn:duality}
\int_{\X \times \Y}c h = -\int_\X u f-\int_\Y v g+\int_{\X \times \Y} w \overline{h}
\end{eqnarray}
we conclude that $h\in\Gamma(f,g)^{\overline{h}}$ is a minimizer of (\ref{eqn: constraind_optimal_cost}) and $(u,v,w)\in \overline{Lip}_c$ a maximizer of (\ref{unconstrained:dual}).\\

We now discuss example~\ref{example} where
$c(x,y)=\frac{1}{2}|x-y|^2$,
$f=g=1|_{[-\frac{1}{2},\frac{1}{2}]}$, and
$\overline{h}=2 \cdot 1|_{[-\frac{1}{2},\frac{1}{2}]^2}$ (figure 1B).\\

Let $u(x):=-\frac{1}{2}x^2$ and $v(y):=-\frac{1}{2}y^2$.
Let $S:=\{(x,y)\in \mathbb{R}^2 \, | \, c(x,y)+u(x)+v(y) \leq 0 \}
=\{(x,y)\in \mathbb{R}^2 \, | \, xy \geq 0 \}$. Note that
$S \cap [-\frac{1}{2},\frac{1}{2}]^2
=[-\frac{1}{2},0] \times [-\frac{1}{2},0] \cup [0, \frac{1}{2}] \times [0, \frac{1}{2}]$.
Now let
$h := \overline{h}|_{S\cap [-\frac{1}{2},\frac{1}{2}]^2} \in \Gamma(f,g)^{\overline{h}}$
(see figure 1A) and let
\begin{equation*}
w(x,y):= \left\{
\begin{array}{cc}
c(x,y)+u(x)+v(y) & \mbox{ on } S\\
0 & \mbox{ on } \mathbb{R}^2 \setminus S.
\end{array}
\right.
\end{equation*}

Since $w(x,y)\leq 0$ on $\mathbb{R}^2$, $c(x,y)+u(x)+v(y)-w(x,y)$ is $=0$ on $S$, and is $>0$ on $\mathbb{R}^2 \setminus S$, $(u,v,w)\in \overline{Lip}_c$. Integrating $w$ against $h$ we get:
\begin{eqnarray*}
& &\int_{\mathbb{R}\times\mathbb{R}}c(x,y)h(x,y) + \int_\mathbb{R} u(x)f(x) + \int_\mathbb{R} v(y)g(y)\\
&=&\int_{\mathbb{R}\times\mathbb{R}}\{c(x,y)+u(x)+v(y)\}h(x,y) = \int_{S\cap [-\frac{1}{2},\frac{1}{2}]^2} w(x,y) h(x,y)\\
&=& \int_{S\cap [-\frac{1}{2},\frac{1}{2}]^2} w(x,y) \overline{h}(x,y)
= \int_{\mathbb{R}\times\mathbb{R}} w(x,y) \overline{h}(x,y).\\
\end{eqnarray*}
That is, the given $h, u, v,$ and $w$ satisfy equation (\ref{eqn:duality}). Hence $h$ minimizes the primal problem, and so is optimal, while $(u,v,w)$ maximizes the dual problem.

\end{document}